\documentclass[a4paper,12pt,twoside]{article}

\usepackage{multirow}

\usepackage{float}
\restylefloat{table}

\usepackage{doi}

\usepackage{etoolbox}

\usepackage[dvipsnames]{xcolor} 

\usepackage{float}

\usepackage{hyperref} 
\hypersetup{breaklinks, colorlinks,
    linkcolor = {Blue},
    citecolor = {Blue},
    urlcolor  = {Blue}
}

\usepackage{amsmath,amssymb}
\usepackage{mathtools} 

\usepackage{amsthm, 
}

\usepackage[numbers,sort&compress]{natbib} 
\usepackage[text={6.5in,9.2in},centering]{geometry}

\usepackage[inline,shortlabels]{enumitem} 
\renewlist{enumerate}{enumerate}{2}
\setlist[enumerate,1]{label=\textup{(\alph*)},
ref={\alph*}, align=left, labelsep=0.5ex, leftmargin=*}
\setlist[enumerate,2]{label=\textup{({\roman*})},
ref={\roman*}, align=right, labelsep*=1ex, widest={(ii)},  
leftmargin=5.4ex}

\DeclareSymbolFont{bbold}{U}{bbold}{m}{n}
\DeclareMathSymbol{\bbS}{\mathord}{bbold}{83}      

\theoremstyle{plain}
\newtheorem{theorem}{Theorem}[section]
\newtheorem{corollary}{Corollary}[section]
\newtheorem{lemma}{Lemma}[section]

\newtheoremstyle{boldremex}
    {\dimexpr\topsep/2\relax} 
    {\dimexpr\topsep/2\relax} 
    {}          
    {}          
    {\bfseries} 
    {.}         
    {.5em}      
    {}          
    
\theoremstyle{boldremex}
\newtheorem{remark}{Remark}[section]

\makeatletter
\renewenvironment{proof}[1][\proofname]{%
   \par\pushQED{\qed}\normalfont%
   \topsep6\p@\@plus6\p@\relax
   \trivlist\item[\hskip\labelsep\bfseries#1\@addpunct{.}]%
   \ignorespaces
}{%
   \popQED\endtrivlist\@endpefalse
}
\makeatother

\newcommand{\CC}{\mathbb{C}}
\newcommand{\RR}{\mathbb{R}}
\newcommand{\NN}{\mathbb{N}}
\newcommand{\ZZ}{\mathbb{Z}}
\newcommand{\Zpl}{\ZZ_+}

\newcommand{\Comp}{\mathrm{Comp}}
\newcommand{\Part}{\mathrm{Part}}
 
\newcommand{\dd}{{\mathrm d}}
\newcommand{\ee}{{\mathrm e}}
\newcommand{\ii}{{\mathrm i}}
\newcommand{\pii}{{\mathrm \pi}}

\newcommand{\abs}[1]{|#1|}                   

\newcommand{\Abs}[1]{\bigl|#1\bigr|}         

\newcommand{\ABS}[1]{\Bigl|#1\Bigr|}         

\newcommand{\card}[1]{|#1|}                  

\newcommand{\EE}{\mathrm{E}}                 

\newcommand{\floor}[1]{{\lfloor #1\rfloor}}

\newcommand{\haf}{\mathrm{haf}} 

\renewcommand{\geq}{\geqslant}

\renewcommand{\leq}{\leqslant}

\newcommand{\norm}[1]{\|#1\|}

\newcommand{\per}{\mathrm{per}}

\newcommand{\set}[1]{\underline{#1}}

\newcommand{\Subs}[2]{\bbS(#1,#2)} 
 
\newcommand{\type}{\mathrm{type}}

\newcommand{\vecsum}[1]{|#1|}

\newcommand{\binomial}[2]{\genfrac{(}{)}{0pt}{}{#1}{#2}}

\newcommand{\newatop}[2]{\genfrac{}{}{0pt}{}{\scriptstyle #1}{\scriptstyle #2}}

\scrollmode
\begin{document}
\title{
New inequalities for permanents and hafnians \\
and some generalizations} 
\author{Bero Roos\footnote{
Postal address: 
FB IV -- Mathematics,
University of Trier,
54286 Trier,
Germany. \newline
E-mail: bero.roos@uni-trier.de
}\\
University of Trier
}

\date{}
\maketitle
\begin{abstract} 
We show new upper bounds for permanents and hafnians, which are 
particularly useful for complex matrices. Multidimensional permanents 
and hyperhafnians are considered as well. The permanental bounds
improve on a Hadamard type inequality of Carlen, Lieb and Loss
(2006, Methods and Applications of Analysis 13, 1--17) and
Cobos, K\"uhn and Peetre (2006, Integral Equations and Operator 
Theory 56, 57--70). 
Our proofs are based on a more general inequality, which can be applied 
to generalized Laplace type expansions of the matrix functions under 
consideration. As application, we show new bounds on the 
characteristic function of random diagonal sums. A
numerical comparison  shows the performance
of some of our permanental bounds. 
\medskip

\noindent 
{\small \emph{2020 Mathematics Subject Classification.}
Primary
15A15;    
secondary 
15A45.    

}
\end{abstract}
\section{Introduction and main results for permanents}
\subsection{Motivation}
The permanent of a square matrix is defined as a kind of ``signless'' 
determinant; a precise definition can be found in \eqref{e76521655} 
below. However, unlike determinants,
permanents are in many cases difficult to evaluate, see 
\citet{MR526203}. Therefore inequalities for 
permanents have been extensively studied in the literature. 
For upper bounds of permanents of non-negative matrices (or matrices 
with entries in $\{0,1\}$), see \citet{MR504978}, 
\citet{MR1752167,MR2010711,MR2100577}, \citet{MR2460503}, 
\citet{MR2382516}, and the references therein.  
Such inequalities can also easily be employed for complex matrices by 
using the triangle inequality, see \cite[Section 6.4]{MR504978}. 
In fact, suppose that $Z=(z_{j,r})$ is a complex $n\times n$ matrix 
for $n\in\NN=\{1,2,3,\dots\}$ and $\per(Z)$ denotes its permanent. 
Then $\abs{\per(Z)}\leq\per(\abs{Z})$, where 
$\abs{Z}=(\abs{z_{j,r}})$. Now any permanental upper bound for 
non-negative matrices can be applied to the right-hand side of the 
inequality above. But the resulting bound for $\abs{\per(Z)}$ 
only depends on the absolute values of the entries of $Z$, that is, on 
$\abs{Z}$. In particular, this holds for an upper 
bound of Hadamard type, which we discuss below, see \eqref{e75870}.

In recent years, there has been increasing interest in permanents of 
complex matrices, see, for example, \citet{MR2115306}, 
\citet{MR3029555}, \citet{MR3464208}, and \citet{MR3899574}.
However, it seems that there are only a few upper bounds of $\per(Z)$ 
available not just depending on $\abs{Z}$. An important one is 
due to \citet{MR0194445}. From their Corollary 3.2, it follows that 
\begin{align}\label{e782649}
\abs{\per(Z)}\leq\Big(\frac{1}{n}\sum_{j=1}^n\alpha_j^{2n}\Big)^{1/2},
\end{align}
where $\alpha_1,\dots,\alpha_n$ are the singular values of $Z$, that 
is, the non-negative square roots of the eigenvalues of $ZZ^*$. 
Here $Z^*$ denotes the complex conjugate transpose of $Z$. 
Another inequality was proved by \citet[Remark 3]{MR1054132} and says 
that  
\begin{align}\label{e6217658}
\abs{\per(Z)}\leq\norm{Z}_p^n,
\end{align}
where $p\in[1,\infty]$ and 
$\norm{Z}_p=\sup\{\norm{Zx}_p\,|\,
x\in\CC^{n\times 1},\,\norm{x}_p\leq 1\}$ 
is the operator norm of $Z$ with respect to the $\ell^p$-norm on 
$\CC^{n\times 1}$. Here $\CC$ is the set of complex numbers,
$\CC^{n\times 1}$ is the vector space of all column $n$-tuples of 
numbers in $\CC$, and $Zx\in\CC^{n\times 1}$ 
denotes the product of matrix $Z$ with vector $x$. For $p=2$, 
\eqref{e6217658} also follows from \eqref{e782649} as has been noted by 
\citet[page 273]{MR761074}. The more general approach in 
\citet[Section 5]{MR2237389} leads to a simple proof  of 
\eqref{e6217658} (see also \citet[Section 2]{MR3235277} for the 
case $p=2$). Unfortunately, the computation of the bounds in 
\eqref{e782649} and \eqref{e6217658} may be somewhat complicated. 

Permanents of matrices with entries in $\{-1,1\}$ can be estimated from 
above by using a bound, which was conjectured by 
\citet[pages 13--14]{MR842720} and recently proved by 
\citet{MR3875614}:
if $n\geq 5$ and if $Z$ has entries in $\{-1,1\}$ and rank $k\in\set{n}$, then 
\begin{align}\label{e8732657}
\abs{\per(Z)}\leq \per(D_{n,k-1}),
\end{align}
where $D_{n,\ell}=(d_{i,j})$ for $\ell\in\{0,\dots,n\}$ is the 
$n\times n$ matrix with entries 
\begin{align*}
d_{i,j}=\left\{
\begin{array}{rl}
-1 & \mbox{ if } i=j\leq \ell,\\
1  & \mbox{ otherwise},
\end{array}
\right.\qquad \mbox{ for }i,j\in\{1,\dots,n\}.
\end{align*}
As has been shown by \citet[page 236]{MR344267}
and mentioned in \cite[Proposition 5.2]{MR842720},
the right-hand side of \eqref{e8732657} can be evaluated using that 
\begin{align*}
\per(D_{n,k-1})=\sum_{j=0}^{k-1}(-2)^j\binomial{k-1}{j}(n-j)!.
\end{align*}

For an inequality of Hadamard type for permanents
and a slightly stronger statement, we need further notation. 
For two sets $J$ and $K$, let $K^J=\{f\,|\,f:\,J\longrightarrow K\}$ 
be the set of all maps from $J$ to $K$
and let $K_{\neq}^J=\{f\in K^J\,|\,f\mbox{ is injective}\}$.   
For instance, if $J$ and $K$ have the same cardinality 
$\card{J}=\card{K}\in\NN$, then $K_{\neq}^J$ is the set of all 
bijective maps from $J$ to $K$. For $n\in\NN$, let 
$\set{n}=\{1,\dots,n\}$, $K^n=K^{\set{n}}$, and 
$K_{\neq}^n=K_{\neq}^{\set{n}}$. 
In particular, $\set{n}_{\neq}^{n}$
is the set of all permutations on the set $\set{n}$. 
If $f\in K^J$, then we write 
$f_j=f(j)$ for $j\in J$ and, in the case $n\in\NN$ and $J=\set{n}$, 
$f=(f_1,\dots,f_n)$. 
The permanent of a matrix $Z=(z_{j,k})\in\CC^{J\times K}$ 
for two finite sets $J$ and $K$ with the same cardinality 
$\card{J}=\card{K}$ is defined by 
\begin{align}\label{e76521655}
\per(Z)=\sum_{k\in K^{J}_{\neq}}\prod_{j\in J}z_{j,k_j}
=\sum_{j\in J^K_{\neq}}\prod_{k\in K}z_{j_k,k}.
\end{align} 
If $J=K=\emptyset$, then  $\per(Z)=1$, since empty products are 
defined to be $1$. 
Permanents are of considerable interest in various areas of 
science. Properties and applications can, for instance,
be found in \citet{MR504978,MR688551,MR900069}, \citet{MR2140290}, 
and \citet{MR3540979}.

An inequality of Hadamard type for the permanent of a matrix 
$Z=(z_{j,r})\in\CC^{n\times n}=\CC^{\set{n}\times\set{n}}$ 
for $n\in\NN$ states that 
\begin{align}\label{e75870}
\abs{\per(Z)}\leq n!\prod_{r=1}^n\Big(\frac{1}{n}\sum_{j=1}^n
\abs{z_{j,r}}^2\Big)^{1/2},
\end{align}
see \citet[Theorem 1.1]{MR2275869} and \citet[Theorem 5.1]{MR2256997}. 
In \citep{MR2275869}, the reader can find two different proofs of 
\eqref{e75870}, the second of which contains a slightly stronger 
statement on permanents of submatrices of $Z$, that is permanental 
minors. For sets $J'\subseteq J$, $K'\subseteq K$ and a matrix 
$Z=(z_{j,k})\in\CC^{J\times K}$, let $Z[J',K']\in\CC^{J'\times K'}$ 
denote the submatrix of $Z$ with entries $z_{j,k}$ for 
$(j,k)\in J'\times K'$. \citet[Theorem 3.1]{MR2275869} proved that, 
for $Z=(z_{j,r})\in\CC^{n\times n}$, $M\subseteq \set{n}$, 
and $m=\card{M}$, 
\begin{align} \label{e21659}
\frac{1}{\binomial{n}{m}}
\sum_{J\in\Subs{\set{n}}{m}}
\ABS{\frac{1}{m!}\per(Z[J,M])}^2
&\leq \prod_{r\in M}\Bigl(\frac{1}{n}
\sum_{j\in\set{n}}\abs{z_{j,r}}^2\Bigr),
\end{align}
where, for a set $K$ and $m\in\Zpl=\{0,1,2,\dots\}$ with 
$m\leq\card{K}$, we denote by
\begin{align*}
\Subs{K}{m}=\{K'\,|\,K'\subseteq K,\,\card{K'}=m\}
\end{align*}
the set of all subsets of $K$ containing exactly $m$ elements.
For $M=\set{n}$ and $m=n$, \eqref{e21659} reduces to \eqref{e75870}. 

The main aim of the present paper is to present inequalities 
better than \eqref{e75870} and \eqref{e21659}, 
where the bounds do not only depend on 
$\abs{Z}$. For instance, one of our results is the following theorem
containing upper bounds of the permanent of a matrix with
entries on the unit circle in the complex plane.
A proof can be found in Section \ref{s08368}. 
Let $\RR$ be the set of real numbers. 
For $x\in\RR$, let $\floor{x}\in\ZZ$ be the largest integer $\leq x$.

\begin{theorem}\label{t28765988}
Let $n\in\NN\setminus\{1\}$, $d=\floor{\frac{n}{2}}$, $t\in\RR$, 
$Z(t)=(z_{j,r}(t))\in\CC^{n\times n}$ with 
$z_{j,r}(t)=\exp(\ii t x_{j,r})$ and 
$x_{j,r}\in\RR$ for all $j,r\in\set{n}$. 
Let $y_{j,k,r,s}=x_{j,r}-x_{k,r}-x_{j,s}+x_{k,s}$ for 
$(j,k),(r,s)\in\set{n}_{\neq}^2$.
For arbitrary 
$s=(s(1),\dots,s(n))\in\set{n}_{\neq}^n$, we then have 
\begin{align}\label{e56134}
\frac{1}{n!}\abs{\per(Z(t))}
\leq \prod_{r=1}^{d}\Big(\frac{1}{n(n-1)}
\sum_{(j,k)\in\set{n}_{\neq}^2}\cos^2\Big(
\frac{ty_{j,k,s(2r-1),s(2r)}}{2}\Big)\Big)^{1/2}.
\end{align}
Further
\begin{align}\label{e56135}
\frac{1}{n!}\abs{\per(Z(t))}
&\leq \Big(\frac{1}{n^2(n-1)^2}\sum_{(r,s)\in\set{n}_{\neq}^2}
\sum_{(j,k)\in\set{n}_{\neq}^2}
\cos^2\Big(\frac{ty_{j,k,r,s}}{2}\Big)\Big)^{d/2}.
\end{align}
\end{theorem}
\begin{remark} \label{r62586}
Let the assumptions of Theorem \ref{t28765988} hold.
\begin{enumerate}

\item \label{r62586.a} In the present situation, the inequalities in 
Theorem \ref{t28765988} are better than \eqref{e75870}, 
since the latter only gives $ \frac{1}{n!}
\abs{\per(Z(t))}\leq1$ while the right-hand sides of \eqref{e56134} and 
\eqref{e56135} are always $\leq 1$. 
The function $\RR\ni t\mapsto \tfrac{1}{n!}\per(Z(t))$ 
can be interpreted as the characteristic function of a linear rank 
statistic; see Section \ref{s676257}, for the discussion of a more 
general situation. Theorem \ref{t28765988} complements 
Theorem 2.1 in \citet{MR719749} under the present assumptions. 
In contrast to the upper bound given in that theorem, 
our bounds contain explicit constants and are valid for all $t\in\RR$.
For a first application of \eqref{e56135}, see \citet{Roos2020}.

\item The bound in \eqref{e56134} depends on an arbitrary permutation 
$s$. For $s=(1,\dots,n)$, we obtain
\begin{align*} 
\frac{1}{n!}\abs{\per(Z(t))}
\leq \prod_{r=1}^{d}\Big(\frac{1}{n(n-1)}
\sum_{(j,k)\in\set{n}_{\neq}^2}\cos^2\Big(
\frac{ty_{j,k,2r-1,2r}}{2}\Big)\Big)^{1/2},
\end{align*}
which however, may not be the best bound obtainable 
from \eqref{e56134}. 
In general, it is incomparable with \eqref{e56135}. 

\item 
From the simple identity 
$\cos^2(x)-1+x^2=4x^2\int_0^1(1-u)\sin^2(ux)\,\dd u$
together with the inequality $\abs{\sin(x)}\leq \abs{x}$ 
for $x\in\RR$, it follows that 
\begin{align*}
\cos^2(x)
\leq 1-x^2+x^2\min\Big\{1,\frac{x^2}{3}\Big\},\qquad(x\in\RR).
\end{align*}
This can be applied to the right-hand sides of the inequalities in 
Theorem \ref{t28765988}. For instance, from \eqref{e56135} we derive
\begin{align}\label{e825766}
\frac{1}{n!}\abs{\per(Z(t))}
\leq (1-\theta)^{d/2},
\end{align}
where 
\begin{align*}
\theta&=\frac{t^2}{4n^2(n-1)^2}\sum_{(r,s)\in\set{n}_{\neq}^2}
\sum_{(j,k)\in\set{n}_{\neq}^2}y_{j,k,r,s}^2
\max\Bigl\{0,1- \frac{t^2y_{j,k,r,s}^2}{12}\Bigr\}.
\end{align*}
The right-hand-side of \eqref{e825766} is bounded by $1$; it is
small if $\theta$ and $d$ are  large.
\end{enumerate}
\end{remark}


The method used in this paper is not only applicable to permanents, but 
also to some other matrix functions, such as multidimensional 
permanents,  hafnians, and hyperhafnians. In fact, we only need that 
the matrix function under consideration satisfies a generalized Laplace 
type expansion, see Lemmata \ref{l486369} and \ref{l324579} below. 
These expansions together with a general inequality given in 
Theorem \ref{t26376566} immediately imply our upper bounds. 
It should be mentioned that our 
main inequalities can be generalized in the case of matrices over a 
complex commutative unital Banach algebra. 
However, we do not follow this idea here. 
For the proof of Theorem \ref{t26376566}, we use Theorem \ref{c72566},
which generalizes Theorem \ref{t983467},
which, in turn, contains a non-trivial generalization of an auxiliary 
inequality in \citet[Proposition 3.1]{MR3916882} on subset 
convolutions of set functions. 
We note that, due to lack of space, we omitted the characterizations 
of equality in our inequalities except for the one presented in 
Theorem~\ref{t983467}, see Remark \ref{r28759}. 

The structure of the paper is as follows. The next subsection 
introduces some further notation on weak compositions of non-negative
integers and ordered weak partitions of sets. Subsection \ref{s08368} 
is devoted to our upper bounds for permanents. 
In Subsections \ref{s676257} and \ref{s49698}, we discuss an 
application to random diagonal sums and give a numerical comparison of 
permanental bounds. The purpose of
Section~\ref{s561476} is to present our bounds for 
multidimensional permanents, hafnians, and hyperhafnians. 
The proofs of the results of Section \ref{s561476} are given
in Section~\ref{s7265755} by using the general Theorem \ref{t26376566}.
The latter theorem is proved in Section~\ref{s64514} with the help of 
an inequality on subset convolutions of set functions. 
Section \ref{s82576} contains all the remaining proofs. 

\subsection{Weak compositions and ordered weak partitions}
This subsection is devoted to further notation. For $n\in\Zpl$, let a 
weak composition of $n$ be an (ordered) family 
$w=(w_1,\dots,w_d)\in\Zpl^d$ for $d\in\NN$, satisfying 
$\vecsum{w}:=\sum_{r=1}^d w_r=n$. The components $w_1,\dots,w_d$ of 
$w$ are called parts of $w$. If a weak composition $w$ of $n$ 
contains $d$ parts, then $w$ is called a weak $d$-composition. For 
$n\in\Zpl$ and $d\in\NN$, let $\Comp(n,d)$ be the set of all weak 
$d$-compositions of $n$. 

For a finite set $M$ with cardinality $\card{M}=m\in\Zpl$,
an ordered weak partition of $M$ is an (ordered) family 
$W=(W_1,\dots,W_d)$ for $d\in\NN$ of pairwise disjoint 
sets $W_1,\dots,W_d\subseteq M$ with
$\bigcup_{r=1}^dW_r=M$. The components $W_1,\dots,W_d$ of  $W$
are called blocks of $W$. Let $\Part(M)$ be the set of ordered weak 
partitions of $M$. The type of $W=(W_1,\dots,W_d)\in\Part(M)$ is 
defined by $\type(W)=(\card{W_1},\dots,\card{W_d})\in\Comp(m,d)$. For 
$d\in\NN$ and $w\in\Comp(m,d)$, let $\Part(M,w)$ be the set of all 
ordered weak partitions $W\in\Part(M)$ with $d$ blocks and 
$\type(W)=w$. We note that $\card{\Part(M,w)}=\frac{m!}{w!}$
as is easily shown. Here, as usual, we let $w!=w_1!\cdots w_d!$
for $w\in\Zpl^d$. Let 
$\Part_d(M)=\bigcup_{w\in\Comp(m,d)}\Part(M,w)$ be 
the set of ordered weak partitions of $M$ consisting of $d$ blocks. 

Whenever we speak of a composition (resp. partition), we mean a weak 
composition of a non-negative integer (resp.\ ordered weak partition
of a finite set) if not specified otherwise. It should be emphasized 
that, in the present paper, compositions (resp.\ partitions) are 
allowed to contain zero parts (resp.\ empty blocks). 

\subsection{Inequalities for permanents} \label{s08368}
In what follows, let $n\in\NN$, $\emptyset\neq M\subseteq\set{n}$, 
$m=\card{M}$, and $Z=(z_{j,r})\in\CC^{\set{n}\times M}$.
For $K\subseteq M$ and $k=\card{K}$, let 
\begin{align}
f(Z,K)
&=\frac{1}{\binomial{n}{k}}\sum_{J\in\Subs{\set{n}}{k}}
\ABS{\frac{1}{k!}\per(Z[J,K])}^2,\label{e918476}\\
\widetilde{f}(Z,K)
&=\frac{1}{\binomial{n}{k}}
\sum_{J\in\Subs{\set{n}}{k}}
\prod_{j\in J}\Big(\frac{1}{k}\sum_{r\in K}\abs{z_{j,r}}^2\Big),
\quad(K\neq\emptyset). \nonumber
\end{align}
Further, we set $\widetilde{f}(Z,\emptyset)=1$. 
For $k\in\{0,\dots,m\}$, let 
\begin{align} \label{e671454}
F(Z,k)
=\frac{1}{\binomial{m}{k}}\sum_{K\in\Subs{M}{k}}f(Z,K)
=\frac{1}{\binomial{m}{k}\binomial{n}{k}}
\sum_{K\in\Subs{M}{k}}\sum_{J\in\Subs{\set{n}}{k}}
\ABS{\frac{1}{k!}\per(Z[J,K])}^2.
\end{align}
We have 
\begin{gather}
f(Z,\emptyset)=\widetilde{f}(Z,\emptyset)=F(Z,0)=1,
\quad f(Z,\{r\})
=\frac{1}{n}\sum_{j\in\set{n}}\abs{z_{j,r}}^2 
= \widetilde{f}(Z,\{r\})\quad\mbox{for }r\in M,\label{e8256}\\
F(Z,1)=\frac{1}{mn}\sum_{r\in M}\sum_{j\in\set{n}}\abs{z_{j,r}}^2,\quad
F(Z,m)=f(Z,M).\label{e8625}
\end{gather}
The permanent of $Z$ can be evaluated by using  $F$ or $f$. 
In fact, if $M=\set{n}$ and $m=n$, then 
\begin{align*}
F(Z,n)=f(Z,\set{n})=\ABS{\frac{1}{n!}\per(Z)}^2.
\end{align*}
In Remark \ref{r276559}(\ref{r276559.c}), one can find a formula
for $f(Z,K)$ in the case $\card{K}=2$ and for $F(Z,2)$. 
An application of the Hadamard type inequality \eqref{e75870} to the 
summands in \eqref{e918476} gives 
\begin{align}\label{e86451}
f(Z,K)\leq \widetilde{f}(Z,K)
\quad\mbox{for all } K\subseteq M. 
\end{align}
Employing the notation above, the inequality 
of \citet[Theorem 3.1]{MR2275869} (see \eqref{e21659}) can be 
reformulated as 
\begin{align}\label{e5352735}
f(Z,M)
\leq \prod_{r\in M}f(Z,\{r\})
=\prod_{r\in M}\widetilde{f}(Z,\{r\}).
\end{align}
Using a generalization of the second method of proof 
in \citet{MR2275869}, it was shown in \citet[Theorem 3.1]{MR3916882}
that, for arbitrary $d\in\set{m}$ and $W\in\Part_d(M)$, we have 
\begin{align}\label{e73766509}
f(Z,M)
&\leq\prod_{r=1}^d\widetilde{f}(Z,W_r)
\end{align}
and, in the case $M=\set{n}$ and $m=n$, 
\begin{align}\label{e73766510}
\abs{\per(Z)}
&\leq n!\prod_{r=1}^d\sqrt{\widetilde{f}(Z,W_r)}.
\end{align}
Unfortunately, the bounds in \eqref{e73766509} and \eqref{e73766510}
only depend on $\abs{Z}$. 
The following theorem contains a substantial improvement, 
as follows from \eqref{e86451}.

\begin{theorem}\label{t45789}
Let $n\in\NN$, $\emptyset\neq M\subseteq\set{n}$, $m=\card{M}$, 
$Z=(z_{j,r})\in\CC^{\set{n}\times M}$, and $f(Z,K)$ for 
$K\subseteq M$ be as in \eqref{e918476}. 
For $k\in\{0,\dots,m\}$, $K\in\Subs{M}{k}$, $d\in\NN$, and 
$W\in\Part_d(K)$, we then have
\begin{align}\label{e84376}
f(Z,K)
&\leq\prod_{r=1}^d f(Z,W_r).
\end{align}
In particular, in the case $K=M=\set{n}$ and $k=m=n$, we have 
\begin{align} \label{e76652}
\Abs{\per(Z)}
&\leq n!\prod_{r=1}^d \sqrt{f(Z,W_r)}.
\end{align}
\end{theorem}
We omit the proof of Theorem \ref{t45789} and also the proofs of 
Remark~\ref{r348769}(\ref{r348769.c}) and Theorem \ref{t932657} 
below, since, in Subsection \ref{s78256}, we give some 
generalizations to multidimensional permanents.
For the corresponding proofs, see Section \ref{s7265755}.
Note that, if $\card{W_1},\dots,\card{W_d}$ are bounded, 
the right-hand sides of \eqref{e84376} and \eqref{e76652} can be 
evaluated in polynomial time in $n$.

\begin{remark}\label{r348769}
Let the assumptions of Theorem \ref{t45789} hold. 
\begin{enumerate}\itemsep-2pt

\item  From \eqref{e84376}, it follows that the set 
function $f(Z,\,\cdot\,)$ is logarithmically subadditive on the power 
set of $M$.

\item  
The right-hand side of \eqref{e84376} depends on the choice of 
$d\in\NN$ and the partition $W\in\Part_d(K)$ of the set $K$. 
It should be observed that the 
finer the partition is, the worse is inequality \eqref{e84376}. 
More precisely, suppose that $d'\in\NN$ and $W'\in\Part_{d'}(K)$ 
is another partition of $K$, which is finer than $W$, that is, for 
every $r'\in\set{d'}$, there is an $r\in\set{d}$ such that 
$W_{r'}'\subseteq W_r$. Then Theorem~\ref{t45789} implies that 
\begin{align*}
f(Z,K)\leq\prod_{r=1}^d f(Z,W_r)\leq\prod_{r'=1}^{d'} f(Z,W_{r'}').
\end{align*}
In particular, \eqref{e5352735} is the 
worst inequality among those given in \eqref{e84376} for $K=M$.

\item \label{r348769.c}
In \eqref{e84376}, equality holds if 
\begin{enumerate}[(i)] \itemsep-2pt

\item a number $r\in\set{d}$ exists such that
$\per(Z[J,W_r])=0$ for all $J\in\Subs{\set{n}}{\card{W_r}}$, or 

\item there are numbers $y_1,\dots,y_d\in\CC$ such that 
$\per(Z[J,W_r])=y_r$ for all $r\in\set{d}$ and 
$J\in\Subs{\set{n}}{\card{W_r}}$. 
\end{enumerate}
\end{enumerate}
\end{remark}

See \citet[Theorem 1]{MR0170901}, for an upper bound of $F(Z,k)$ 
depending on the singular values of the complex square matrix $Z$. 
In the following theorem, we present an upper bound, which however is 
difficult to compare with the one mentioned above. 
\begin{theorem}\label{t932657}
Let $n\in\NN$, $\emptyset\neq M\subseteq\set{n}$, $m=\card{M}$,
$Z=(z_{j,r})\in\CC^{\set{n}\times M}$, 
$f(Z,K)$ for $K\subseteq M$ be as in \eqref{e918476}, and 
$F(Z,k)$ for $k\in\{0,\dots,m\}$ be as in \eqref{e671454}. 
For $k\in\{0,\dots,m\}$, $d\in\NN$, and $w\in\Comp(k,d)$, we then have
\begin{align}\label{e84377}
F(Z,k) 
&\leq \prod_{r=1}^dF(Z,w_r).
\end{align}
In particular, in the case $k=m$, we have 
\begin{align} \label{e28376} 
f(Z,M) 
&\leq \prod_{r=1}^dF(Z,w_r).
\end{align}
If $M=\set{n}$ and $k=m=n$, then we obtain
\begin{align}\label{eq875987}
\Abs{\per(Z)}
&\leq n!\prod_{r=1}^d\sqrt{F(Z,w_r)}.
\end{align}
\end{theorem}
Note that, if $w_1,\dots,w_d$ are bounded, 
the right-hand sides of \eqref{e84377}--\eqref{eq875987} can be 
evaluated in polynomial time in $n$.


\begin{remark}\label{r276559}  
Let the assumptions of Theorem \ref{t932657} hold. 
\begin{enumerate} 

\item 
From \eqref{e84377}, it follows that the function 
$F(Z,\,\cdot\,)$ is logarithmically subadditive on the set 
$\{0,\dots,m\}$.

\item \label{r276559.b} 
In \eqref{e84377}, equality holds if all the $z_{j,r}$ for 
$j\in\set{n}$, $r\in M$ are identical. 

\item \label{r276559.c} 
If  $2\leq m\leq n$, $(u,v)\in M^2_{\neq}$, and $K=\{u,v\}$,
then we have $\card{K}=2$ and, as is easily shown,
\begin{align}
f(Z,K)
&
=\frac{1}{4n(n-1)}
\sum_{(j,k)\in\set{n}_{\neq}^2}\abs{z_{j,u}z_{k,v}+z_{k,u}z_{j,v}}^2
\label{e782568}\\
&=\frac{1}{4n(n-1)}\sum_{(j,k)\in\set{n}_{\neq}^2}
\Big(b_{j,k,u,v}
+4a_{j,u}a_{k,v}a_{k,u}a_{j,v}
\cos^2\Big(\frac{y_{j,k,u,v}}{2}\Big)\Big),\nonumber
\end{align} 
where $z_{j,r}=a_{j,r}\exp(\ii x_{j,r})$, $a_{j,r}\in[0,\infty)$,
$x_{j,r}\in\RR$ for $j\in\set{n}$ and $r\in M$, and, for 
$(j,k)\in\set{n}_{\neq}^2$ and $(r,s)\in M_{\neq}^2$,  
\begin{align*} 
b_{j,k,r,s}
&=(a_{j,r}a_{k,s}-a_{k,r}a_{j,s})^2, \quad
y_{j,k,r,s}
=x_{j,r}-x_{k,r}-x_{j,s}+x_{k,s}.
\end{align*}
Further
\begin{align}
F(Z,2)
&=\frac{1}{4m(m-1)n(n-1)}
\sum_{(r,s)\in M_{\neq}^2}\sum_{(j,k)\in\set{n}_{\neq}^2}
\abs{z_{j,r}z_{k,s}+z_{k,r}z_{j,s}}^2\label{e561548}\\
&=\frac{1}{4m(m-1)n(n-1)}
\sum_{(r,s)\in M_{\neq}^2}\sum_{(j,k)\in\set{n}_{\neq}^2}
\Big(b_{j,k,r,s}+4a_{j,r}a_{k,s}a_{k,r}a_{j,s}
\cos^2\Big(\frac{y_{j,k,r,s}}{2}\Big)\Big).\nonumber
\end{align}
If, for example, $a_{j,r}=1$ for all $j\in\set{n}$ and $r\in M$, then
\begin{align}
f(Z,K)
&=\frac{1}{n(n-1)}\sum_{(j,k)\in\set{n}_{\neq}^2}
\cos^2\Big(\frac{y_{j,k,u,v}}{2}\Big),\label{e42648}\\
F(Z,2)
&=\frac{1}{m(m-1)n(n-1)}
\sum_{(r,s)\in M_{\neq}^2}\sum_{(j,k)\in\set{n}_{\neq}^2}
\cos^2\Big(\frac{y_{j,k,r,s}}{2}\Big).
\label{e43515}
\end{align}
\end{enumerate}
\end{remark}

\begin{proof}[Proof of Theorem \ref{t28765988}] 
For $r\in\set{d}$, let $W_r=\{s(2r-1),s(2r)\}$. 
If $n$ is even, then let $W=(W_1,\dots,W_d)\in\Part_d(\set{n})$,
and if $n$ is odd, then let $W_{d+1}=\{s(n)\}$ and 
$W=(W_1,\dots,W_{d+1})\in\Part_{d+1}(\set{n})$. 
Inequality \eqref{e56134} now follows from \eqref{e76652}, 
\eqref{e42648}, and \eqref{e8256}.
Inequality \eqref{e56135} is a consequence of 
\eqref{eq875987}, \eqref{e43515}, and \eqref{e8625}. 
\end{proof}
\begin{remark} 
Let the assumptions of Theorem \ref{t932657} hold. 
For $k\in\{0,\dots,m\}$, let
\begin{align*}
\varphi(Z,k)
=\sum_{K\in\Subs{M}{k}}\sum_{J\in\Subs{\set{n}}{k}}\per(Z[J,K])
\end{align*}
be the sum of all permanental minors of order $k$ of $Z$.
In particular, we have
$\varphi(Z,m)=\sum_{J\in\Subs{\set{n}}{m}}\per(Z[J,M])$. 
A simple application of the Cauchy-Schwarz inequality gives,
for all $k\in\{0,\dots,m\}$, 
\begin{align}\label{e89326}
\abs{\varphi(Z,k)}
\leq \binomial{m}{k}\binomial{n}{k}k!\sqrt{F(Z,k)},\quad  
\abs{\varphi(Z,m)}&\leq \binomial{n}{m}m!\sqrt{f(Z,M)}.
\end{align}
The right-hand sides of the inequalities in \eqref{e89326} can now be 
further estimated by using any upper bound for $F(Z,k)$ or $f(Z,M)$ 
given above. It seems to be difficult to give a detailed comparison 
of the resulting bounds with those from the literature. For 
inequalities concerning $\varphi(Z,k)$ for non-negative square 
matrices, see \citet{MR0194444}, \citet{MR1176450}, \citet{MR1404175}, 
\citet{MR2460503}, and the references given there. 
\end{remark}

\subsection{Application to random diagonal sums}\label{s676257}
Let $n\in\NN\setminus\{1\}$, $d=\floor{\frac{n}{2}}$, 
and $X=(X_{j,r})$  be a random $n\times n$ matrix consisting of 
real-valued random variables $X_{j,r}$ for $j,r\in\set{n}$ 
with characteristic functions $\varphi_{j,r}$,
that is, $\varphi_{j,r}(t)=\EE\ee^{\ii t X_{j,r}}$ for $t\in\RR$, 
where $\EE$ means expectation. We assume that, for every 
$r\in\set{n}_{\neq}^n$, the generalized diagonal 
$(X_{1,r(1)},\dots,X_{n,r(n)})$ of $X$ is stochastically
independent. For instance, this is true
if the family of all rows (or all columns) of $X$ is independent. 

Let $\pi=(\pi(1),\dots,\pi(n))$ be a uniformly distributed random 
permutation of the set $\set{n}$ independent of $X$. Let 
\begin{align*}
S_n=\sum_{j=1}^nX_{j,\pi(j)}
\end{align*}
be the random diagonal sum of $X$, that is, the sum of 
the entries in the generalized random diagonal 
$(X_{1,\pi(1)},\dots,X_{n,\pi(n)})$ of $X$. 
If the entries of $X$ are constants, 
$S_n$ is a linear rank statistics, for instance, see 
\citet{MR1680991}. For the normal approximation of $S_n$ in the 
case that all $X_{j,r}$, $(j,r\in\set{n})$ are independent, 
see \citet{MR3322321} and the references therein. 
In the case that $X_{j,1}=\dots=X_{j,n}$ for all $j\in\set{n}$, 
$S_n=\sum_{j=1}^nX_{j,1}$ is a sum of independent random variables. 

Let $\varphi$ be the characteristic function of $S_n$. It is easily 
seen that, for $t\in\RR$, 
\begin{align*}
\varphi(t)=\EE \ee^{\ii tS_n}=\frac{1}{n!}\per(Z(t)),
\end{align*}
where $Z(t)=(\varphi_{j,r}(t))\in\CC^{n\times n}$. 
The following theorem contains new upper bounds for $\abs{\varphi}$.
\begin{theorem} 
Let the above assumptions hold, $t\in\RR$, and 
$s=(s(1),\dots,s(n))\in\set{n}_{\neq}^n$ be arbitrary. Then 
\begin{align}\label{e87256}
\abs{\varphi(t)} 
&\leq
\prod_{r=1}^d\Bigl(\frac{1}{n(n-1)}
\sum_{(j,k)\in \set{n}_{\neq}^2}
\frac{1}{4}
\abs{\varphi_{j,s(2r-1)}(t)\varphi_{k,s(2r)}(t)
+\varphi_{k,s(2r-1)}(t)\varphi_{j,s(2r)}(t)}^2\Bigr)^{1/2}
\end{align}
and
\begin{align}\label{e987266}
\abs{\varphi(t)} 
&\leq \Bigl(\frac{1}{n^2(n-1)^2}
\sum_{(j,k)\in \set{n}_{\neq}^2}\sum_{(r,s)\in \set{n}_{\neq}^2}
\frac{1}{4}\abs{\varphi_{j,r}(t)\varphi_{k,s}(t)
+\varphi_{k,r}(t)\varphi_{j,s}(t)}^2\Bigr)^{d/2}.
\end{align}
In the case of odd $n$, \eqref{e87256} 
remains true in a sharper form if the right-hand side is multiplied by  
$(\frac{1}{n}\sum_{j=1}^n\abs{\varphi_{j,s(n)}(t)}^2)^{1/2}$.
An analogous statement holds for \eqref{e987266} and the factor 
$(\frac{1}{n^2}\sum_{j=1}^n\sum_{r=1}^n
\abs{\varphi_{j,r}(t)}^2)^{1/2}$.
\end{theorem}
\begin{proof} Let $W$ be as in the proof of Theorem \ref{t28765988}. 
Inequality \eqref{e87256} follows from \eqref{e76652},
\eqref{e782568}, and \eqref{e8256}. 
Inequality \eqref{e987266} is a consequence of
\eqref{eq875987}, \eqref{e561548}, and \eqref{e8625}. 
The additional statement is clear. 
\end{proof}

Upper bounds for characteristic functions can be useful in the 
approximation of probability distributions, nonparametric statistics, 
and stability problems; for instance, see \citet{MR1353441} and 
\citet{MR1745554}. The theorem 
above is a generalization of Theorem \ref{t28765988} and can easily be 
generalized to the case of multivariate random variables $X_{j,r}$ for 
$j,r\in\set{n}$. We could not find any comparable inequality in the 
literature; see also Remark \ref{r62586}(\ref{r62586.a}).  
\begin{remark}
The bounds in \eqref{e87256} and \eqref{e987266} can be viewed as the 
square roots of non-negative characteristic functions of 
sums of independent random variables. This easily 
follows from the fact that, for $(j,k),(r,s)\in\set{n}_{\neq}^2$, 
$\RR\ni t\mapsto\frac{1}{4}\abs{\varphi_{j,r}(t)\varphi_{k,s}(t)
+\varphi_{k,r}(t)\varphi_{j,s}(t)}^2$ 
is the characteristic function of the symmetrized random variable
$Y_{j,k,r,s}-\widetilde{Y}_{j,k,r,s}$, where 
$\widetilde{Y}_{j,k,r,s}$ is an independent copy of 
\begin{align*}
Y_{j,k,r,s}:=I_{j,k,r,s}(X_{j,r}+X_{k,s})+
(1-I_{j,k,r,s})(X_{k,r}+X_{j,s})
\end{align*}
and $I_{j,k,r,s}$ is a Bernoulli random variable independent of 
$X_{j,r},X_{k,s},X_{k,r},X_{j,s}$, which takes the values $0$ and $1$
each with probability $\frac{1}{2}$. 
\end{remark}


\subsection{Numerical example} \label{s49698}
Let us illustrate the performance of some upper bounds. 
Let $n=8$, $t\in\RR$ and 
\begin{align*}
Z=Z(t)
=\begin{psmallmatrix}
1&\ee^{\ii t}&1&1&1&\ee^{\ii t}&1&\ee^{\ii t}\\
1&1&\ee^{\ii t}&\ee^{\ii t}&1&1&\ee^{\ii t}&1\\
\ee^{\ii t}&\ee^{\ii t}&\ee^{\ii t}&1&\ee^{\ii t}
&\ee^{\ii t}&\ee^{\ii t}&1\\
1&\ee^{\ii t}&\ee^{\ii t}&\ee^{\ii t}&1&\ee^{\ii t}&1&\ee^{\ii t}\\
\ee^{\ii t}&1&1&1&1&1&1&\ee^{\ii t}\\
\ee^{\ii t}&\ee^{\ii t}&1&\ee^{\ii t}&1&\ee^{\ii t}&1&\ee^{\ii t}\\
\ee^{\ii t}&1&\ee^{\ii t}&1&\ee^{\ii t}&\ee^{\ii t}&\ee^{\ii t}&1\\
1&1&\ee^{\ii t}&\ee^{\ii t}&1&\ee^{\ii t}&1&\ee^{\ii t}\\
\end{psmallmatrix}    
\in\CC^{n\times n}
\end{align*}
be a matrix with entries on the unit circle in the complex plane.
From \eqref{e76652} and \eqref{eq875987}, it follows that
\begin{align}\label{e165487}
\abs{\per(Z(t))}\leq n!h_k(t) 
\end{align}
and 
\begin{align}\label{e165488}
\abs{\per(Z(t))}\leq n!H_k(t) 
\end{align}
for  $k\in\set{2}$, where 
\begin{align*}
h_1(t)&=(f(Z(t),\{1,2\})f(Z(t),\{3,4\})
f(Z(t),\{5,6\})f(Z(t),\{7,8\}))^{1/2},\\\
h_2(t)&=(f(Z(t),\{1,2,3\})f(Z(t),\{4,5,6\})f(Z(t),\{7,8\}))^{1/2},  
\end{align*}
and
\begin{align*}
H_1(t)=(F(Z(t),2))^{2},\quad
H_2(t)=F(Z(t),3)\sqrt{F(Z(t),2)}.
\end{align*}
Using the computer algebra software Maple, we obtain that 
\begin{align*}
\lefteqn{\frac{1}{n!}\abs{\per(Z(t))}
=\frac{1}{5040} 
(154450
+1145926\cos(t)
+3615364 \cos^2(t)
+6353620 \cos^3(t)}\\
&\quad{}
+6849754 \cos^4(t)
+4692814 \cos^5(t)
+2023768 \cos^6(t)
+508240 \cos^7(t)
+57664 \cos^8(t)
)^{1/2}
\end{align*}
and
\begin{align*}
f(Z(t),\{1,2\})
&=\frac{1}{7}(4+2\cos(t)+\cos^2(t)),\\
f(Z(t),\{3,4\})
&=\frac{1}{56}(37+15\cos(t)+4\cos^2(t)),\\
f(Z(t),\{5,6\})
&=\frac{1}{7}(5+2\cos(t)),\\
f(Z(t),\{7,8\})
&=\frac{1}{28}(13+15\cos^2(t)),\\
f(Z(t),\{1,2,3\})
&=\frac{1}{126}(28+43\cos(t)+43\cos^2(t)+12\cos^3(t)),\\
f(Z(t),\{4,5,6\})
&=\frac{1}{126}(37+50\cos(t)+35\cos^2(t)+4\cos^3(t)),\\
F(Z(t),2)
&=\frac{1}{1568}(963+377\cos(t)+228\cos^2(t)),\\
F(Z(t),3)
&=\frac{1}{14112}(4415+5069\cos(t)+3959\cos^2(t)+669\cos^3(t)).
\end{align*}
We note that the right-hand side of \eqref{e56134} with 
$s=(1,\dots,8)$ (resp.\ the right-hand side of \eqref{e56135}) 
is equal to $h_1(t)$ (resp.\ to $H_1(t)$). 
In Table 1, we compare \eqref{e165487} and \eqref{e165488}
with the bounds from the introduction.

\begin{center}
\begin{tabular}{cc|ccc}
\multicolumn{5}{l}{Table 1: Numerical comparison of bounds and exact 
value}\\ \hline
\multirow{2}{*}{formula number} & \multirow{2}{*}{parameter}
&\multicolumn{3}{c}{upper bounds divided by $n!$}\\  
\multicolumn{2}{c|}{} & 
\multicolumn{1}{c}{$t=\pii$} & 
\multicolumn{1}{c}{$t=\tfrac{\pii}{2}$} & 
\multicolumn{1}{c}{$t=\tfrac{\pi}{4}$}\\ \hline 
\eqref{e6217658} & $p=1$  & $416.1016$ & $416.1016$  & $416.1016$\\
\eqref{e6217658} & $p=\infty$ & $416.1016$ & $416.1016$ & $416.1016$\\
\eqref{e6217658} & $p=2$  & $11.80801$ & $53.71852$ & $250.8386$\\
\eqref{e782649}  & n.a.   &$4.194852$  & $18.99307$ & $88.68481$\\
\eqref{e75870}   &  n.a.   & $1$        & $1$        & $1$ \\
\eqref{e165487}  & $k=1$  & $0.292023$ & $0.353848$ & $0.708592$\\
\eqref{e165488}  & $k=1$  & $0.269499$ & $0.377191$ & $0.734234$\\
\eqref{e8732657} & n.a.    & $0.212699$ & n.a.  & n.a.  \\
\eqref{e165487}  & $k=2$  & $0.134688$ & $0.174062$ & $0.595132$ \\
\eqref{e165488}  & $k=2$  & $0.134585$ & $0.245179$ & $0.670075$ \\
\hline\hline
\multicolumn{2}{c|}{term to be estimated} & 
\multicolumn{3}{c}{exact values} \\ \hline
\multicolumn{2}{c|}{$\tfrac{1}{n!}\abs{\per(Z(t))}$} 
& $0.003968\dots$ & $0.077976\dots$ & $0.556344\dots$ \\ \hline
\end{tabular}
\end{center}

All bounds have been rounded up. The entry ``n.a.'' means 
``not available''. 
In view of the trivial bound $\tfrac{1}{n!}\abs{\per(Z(t))}\leq 1$, 
we see that \eqref{e6217658} for $p\in\{1,2,\infty\}$, 
\eqref{e782649}, and \eqref{e75870} do not give any further 
information. For \eqref{e8732657}, we used that 
$Z(\pi)\in\{-1,1\}^{8\times 8}$ has rank $7$. 
The best bounds in this example are \eqref{e165487} and 
\eqref{e165488} with parameter $k=2$. In particular, 
we see that $h_1(t)$ and $H_1(t)$ 
(resp.\ $h_2(t)$ and $H_2(t)$) are generally incomparable, since $h_1(\pii)>H_1(\pii)$, but $h_1(t)< H_1(t)$ for 
$t\in\{\frac{\pii}{2},\frac{\pii}{4}\}$, and analogously for 
$h_2(t)$ and $H_2(t)$.  
\section{Further results}\label{s561476}
\subsection{Inequalities for multidimensional permanents}\label{s78256}
For $\ell\in\NN$, finite sets $J_1,\dots,J_\ell$, $K$ with 
$\card{J_1}=\dots=\card{J_\ell}=\card{K}$, 
and an $(\ell+1)$-dimensional matrix 
$Z=(z(j_1,\dots,j_\ell,k))
\in\CC^{J_1\times\dots\times J_\ell\times K}$,
the $(\ell+1)$-dimensional permanent of $Z$ is defined by 
\begin{align}\label{e8612766}
\per_\ell(Z)
=\sum_{j^{(1)}\in (J_1)_{\neq}^K}\dots\sum_{j^{(\ell)}\in 
(J_\ell)^K_{\neq}}\prod_{k\in K}z(j_k^{(1)},\dots,j_k^{(\ell)},k).
\end{align} 
Two-dimensional permanents are permanents as defined in 
\eqref{e76521655}, that is $\per_1(Z)=\per(Z)$. 
Sometimes $(\ell+1)$-dimensional permanents are called 
$(\ell+1)$-way permanents, see \citet{Rice1918} and 
\citet[Chapter XXIV]{MR0114826} or hyperpermanents, see 
\citet{MR2422309} and  \citet{ShashuaZassHazan2006}. 
For properties and applications of multidimensional permanents, see
\citet{MR884116}, \citet[Chapter 4]{MR3558532},  
\citet{MR3581885}, and the references therein.
The following Theorems \ref{t45790} and \ref{t257876} contain our 
inequalities for multidimensional permanents. 
\begin{theorem}\label{t45790}
Let $\ell,n\in\NN$, 
$\emptyset\neq M\subseteq\set{n}$, $m=\card{M}$, 
$Z=(z(j_1,\dots,j_\ell,r))
\in\CC^{\set{n}\times\dots\times\set{n}\times M}$. 
For a set $K\subseteq M$ and $k=\card{K}$, let 
\begin{align}\label{e782566}
f_\ell(Z,K)
&=\frac{1}{\binomial{n}{k}^\ell}\sum_{J_1\in\Subs{\set{n}}{k}}
\dots\sum_{J_\ell\in\Subs{\set{n}}{k}}
\ABS{\frac{1}{(k!)^\ell}\per_\ell(Z[J_1,\dots,J_\ell,K])}^2.
\end{align}
For $k\in\{0,\dots,m\}$, $K\in\Subs{M}{k}$, $d\in\NN$,
and $W\in\Part_d(K)$, we then have 
\begin{align}\label{e84378}
f_\ell(Z,K) &\leq \prod_{r=1}^d f_\ell(Z,W_r).
\end{align}
In particular, in the case $K=M=\set{n}$ and $k=m=n$, we have 
\begin{align*}  
\Abs{\per_\ell(Z)} &\leq  (n!)^\ell \prod_{r=1}^d \sqrt{f_\ell(Z,W_r)}.
\end{align*}
\end{theorem}

\begin{remark}\label{r2875876}
In \eqref{e84378}, equality holds if 
\begin{enumerate}[(i)] 

\item \label{r2875876.i} a number $r\in\set{d}$ exists such that
$\per_\ell(Z[J_1,\dots,J_\ell,W_r])=0$ for all 
$J_1,\dots,J_\ell\in\Subs{\set{n}}{\card{W_r}}$, or 

\item \label{r2875876.ii}
there are numbers $y_1,\dots,y_d\in\CC$ such that 
$\per_\ell(Z[J_1,\dots,J_\ell,W_r])=y_r$ for all $r\in\set{d}$ and 
$J_1,\dots,J_\ell\in\Subs{\set{n}}{\card{W_r}}$. 
\end{enumerate}
\end{remark}

\begin{theorem}\label{t257876}
Let $\ell,n\in\NN$, $\emptyset\neq M\subseteq\set{n}$, $m=\card{M}$,
$Z=(z(j_1,\dots,j_\ell,r))
\in\CC^{\set{n}\times\dots\times\set{n}\times M}$, 
and, for $k\in\{0,\dots,m\}$,
\begin{align} \label{t257877}
F_\ell(Z,k)
&=\frac{1}{\binomial{m}{k}}\sum_{K\in\Subs{M}{k}} f_\ell(Z,K),
\end{align}
where $f_\ell(Z,K)$ for $K\in\Subs{M}{k}$ is defined as in 
\eqref{e782566}. For $k\in\{0,\dots,m\}$, $d\in\NN$, and
$w\in\Comp(k,d)$, we then have
\begin{align}\label{t257878}
F_\ell(Z,k) 
&\leq \prod_{r=1}^dF_\ell(Z,w_r).
\end{align}
In particular, in the case $k=m$, we have 
\begin{align*} 
f_\ell(Z,M) 
&\leq \prod_{r=1}^dF_\ell(Z,w_r).
\end{align*}
If $M=\set{n}$ and $k=m=n$, then we obtain
\begin{align*} 
\Abs{\per_\ell(Z)}
&\leq (n!)^\ell\prod_{r=1}^d\sqrt{F_\ell(Z,w_r)}.
\end{align*}
\end{theorem}

\begin{remark}\label{r8164876}
In \eqref{t257878}, equality holds if all the 
$z(j_1,\dots,j_\ell,r)$ for $j_1,\dots,j_\ell\in\set{n}$, 
$r\in M$ are identical. 
\end{remark}

For $\ell=1$, Theorems \ref{t45790}, \ref{t257876} and 
Remarks \ref{r2875876}, \ref{r8164876}
simplify to Theorems \ref{t45789}, \ref{t932657} and 
Remarks \ref{r348769}(\ref{r348769.c}), 
\ref{r276559}(\ref{r276559.b}), 
respectively. The proofs of Theorems \ref{t45790}, \ref{t257876}
and Remark \ref{r2875876} can be found in Section \ref{s7265755}.

\subsection{Inequalities for hafnians}
Let $m\in\NN$, $J$ be a set with $\card{J}=n=2m$, 
$Z=(z_{j,r})\in\CC^{J\times J}$ be a symmetric matrix, that is 
$z_{j,r}=z_{r,j}$
for all $j,r\in J$. The hafnian of $Z$, 
introduced by the physicist \citet{MR0059787}, is defined by  
\begin{align*}
\haf(Z)=\sum_{\{\{j_1,k_1\},\dots,\{j_m,k_m\}\}\in T}
\prod_{r=1}^m z_{j_r,k_r},
\end{align*}
where T is the set of all $\frac{n!}{m!2^m}$
unordered partitions of $J$ into unordered pairs. 
Alternatively, 
\begin{align}\label{e276547}
\haf(Z)
=\frac{1}{m!2^m}\sum_{j\in J_{\neq}^n}\prod_{r=1}^m z_{j(2r-1),j(2r)}.
\end{align}
For convenience, we set $\haf(Z)=1$ 
for $Z\in\CC^{J\times J}$ with $J=\emptyset$. We note that 
$\haf(Z)$ is independent of the values $z_{j,j}$ for $j\in\set{n}$.  
Hafnians are generalizations of permanents since, for 
$Z\in\CC^{n\times n}$, we have  
\begin{align*} 
\per(Z)=
\haf\begin{pmatrix}
0& Z\\
Z^T & 0
\end{pmatrix},
\end{align*}
where $Z^T$ is the transpose of $Z$. 
This and further properties of hafnians can be found, e.g., 
in \citet[Chapter 4]{MR3558532} and \citet{MR0464973}.
The following theorem contains our inequality for hafnians.  
This is generalized to hyperhafnians in Theorem \ref{t7832577}, 
which is proved in Section~\ref{s7265755}.

\begin{theorem}\label{t78368}
Let $m\in\NN$, $n=2m$, $Z=(z_{j,r})\in\CC^{n\times n}$ be a symmetric 
matrix, and 
\begin{align*} 
G(Z,k)
&=\frac{1}{\binomial{n}{2k}}\sum_{J\in\Subs{\set{n}}{2k}}
\ABS{\frac{k!2^k}{(2k)!} \haf(Z[J,J])}^2 \quad \mbox{ for }
k\in\{0,\dots,m\}.
\end{align*}
For $k\in\{0,\dots,m\}$, $d\in\NN$, and $w\in\Comp(k,d)$, we then have 
\begin{align}\label{e51647}
G(Z,k)
\leq\prod_{r=1}^d G(Z,w_r).
\end{align} 
In particular, if $k=m$, then
\begin{align}\label{e86256}
\abs{\haf(Z)}
\leq\frac{n!}{m!2^m}\prod_{r=1}^d\sqrt{G(Z,w_r)}.
\end{align} 
\end{theorem}
Note that, if $w_1,\dots,w_d$ are bounded, 
the right-hand sides of \eqref{e51647} and \eqref{e86256} can be 
evaluated in polynomial time in $n$.
\begin{remark}\label{r2814576} 
Let the assumptions of Theorem \ref{t78368} hold.
\begin{enumerate} \itemsep-2pt

\item From \eqref{e51647}, it follows that the function 
$G(Z,\,\cdot\,)$ is logarithmically subadditive on the set 
$\{0,\dots,m\}$.

\item In \eqref{e51647}, equality holds if all the 
$z_{j,r}$ for $(j,r)\in\set{n}_{\neq}^2$ are identical.

\item \label{r2814576.c}
We have 
\begin{align}
G(Z,1)
&=\frac{1}{n(n-1)}\sum_{(j,r)\in\set{n}_{\neq}^2}\abs{z_{j,r}}^2,
\label{e6721456}\\
G(Z,2)
&=\frac{(n-4)!}{n!}\sum_{(u,v,w,x)\in\set{n}_{\neq}^4}\ABS{
\frac{1}{3}(z_{u,v}z_{w,x}+z_{u,w}z_{v,x}+z_{u,x}z_{v,w})}^2
\quad \mbox{ if }n\geq 4.\label{e672545}
\end{align}
\end{enumerate}
\end{remark}

The following corollary is a simple consequence of Theorem \ref{t78368}
and Remark \ref{r2814576}(\ref{r2814576.c}). 
\begin{corollary} 
Let $m\in\NN\setminus\{1\}$, $n=2m$, 
$Z=(z_{j,r})\in\CC^{n\times n}$ be a symmetric matrix. 
If $m=2d+\ell$ for $d,\ell\in\Zpl$, then  
\begin{align*}
\begin{split}
\abs{\haf(Z)}
&\leq \frac{n!}{m!2^m}\Big(\frac{(n-4)!}{n!}
\sum_{(u,v,w,x)\in\set{n}_{\neq}^4}\ABS{\frac{1}{3}
(z_{u,v}z_{w,x}+z_{u,w}z_{v,x}+z_{u,x}z_{v,w})}^2\Big)^{d/2}\\
&\quad{}\times\Big(\frac{1}{n(n-1)}\sum_{(j,r)\in\set{n}_{\neq}^2}
\abs{z_{j,r}}^2\Big)^{\ell/2}.
\end{split}
\end{align*}
In particular, if $d=0$ and $m=\ell$, we have 
\begin{align}\label{e56146}
\abs{\haf(Z)}
&\leq
\frac{n!}{m!2^m}\Big(\frac{1}{n(n-1)}
\sum_{(j,r)\in\set{n}_{\neq}^2}\abs{z_{j,r}}^2\Big)^{m/2}.
\end{align} 

\end{corollary}

\begin{remark}  
Let the assumptions of Theorem \ref{t78368} hold. 
\begin{enumerate}

\item 
The inequality given in \citet[Theorem 1]{MR0291000} implies that 
\begin{align}\label{e645617}
\abs{\haf(Z)}\leq \sqrt{\per(\abs{Z})},
\end{align}
where $\abs{Z}=(\abs{z_{j,r}})\in[0,\infty)^{n\times n}$ as previously.
This inequality is not easily comparable with \eqref{e86256}. 
Numerical examples show that it is sometimes better, but 
sometimes worse than \eqref{e56146}. 
Further, if $z_{j,r}=y\neq0$ for all $(j,r)\in\set{n}_{\neq}^2$
and $z_{j,j}=0$ for all $j\in\set{n}$, then,
in \eqref{e86256}, equality holds, but the 
right-hand side of \eqref{e645617} is larger than 
the left-hand side by a factor of the order $n^{1/4}$, 
as is easily shown by using the Stirling formula
and the identities $\haf(Z)=\frac{n!}{m!2^m}y^m$ and   
$\per(Z)=n!y^n\sum_{j=0}^n\frac{(-1)^j}{j!}$, 
see \citet[page 44]{MR504978}.

\item 
For $k\in\{0,\dots,m\}$, let
\begin{align*}
\psi(Z,k)=\sum_{J\in\Subs{\set{n}}{2k}} \haf(Z[J,J])  
\end{align*}
denote the sum of all subhafnians of order $2k$ of $Z$. 
The Cauchy-Schwarz inequality implies that  
\begin{align}
\abs{\psi(Z,k)}
&\leq\binomial{n}{2k}\frac{(2k)!}{k!2^k}\sqrt{G(Z,k)}. \label{e782587}
\end{align}
Now inequalities for $G(Z,k)$ can be used to give upper bounds of 
the right-hand side of \eqref{e782587}. 
For instance, \eqref{e51647} and \eqref{e6721456} imply that
\begin{align}
\abs{\psi(Z,k)}
\leq\binomial{n}{2k}\frac{(2k)!}{k!2^k}\Big(
\frac{1}{n(n-1)}\sum_{(j,r)\in\set{n}_{\neq}^2}\abs{z_{j,r}}^2
\Big)^{k/2}. \label{e1287597}
\end{align}
In \eqref{e1287597}, equality holds if all the 
$z_{j,r}$ for $(j,r)\in\set{n}_{\neq}^2$ are identical.
A further upper bound can immediately be written down by applying
\eqref{e51647}, \eqref{e6721456}, and \eqref{e672545}.
\end{enumerate}
\end{remark}

\subsection{Inequalities for hyperhafnians} 
The last result of this section contains an inequality for 
hyperhafnians. Let $\ell\in\NN$ with $\ell\geq2$, $m\in\NN$, $J$ be a 
set with $\card{J}=n=\ell m$, 
$Z=(z({j_1,\dots,j_\ell}))\in\CC^{J\times\dots\times J}$ be 
an $\ell$-dimensional symmetric matrix, that is $z(j_1,\dots,j_\ell)$ 
is invariant under permutations of $j_1,\dots,j_\ell\in J$. Then the 
hyperhafnian of $Z$ can be defined by 
\begin{align}\label{e762576876}
\haf_\ell(Z)
=\frac{1}{m!(\ell!)^m}\sum_{j\in J_{\neq}^n}\prod_{r=0}^{m-1} 
z(j(r\ell+1),\dots,j(r\ell+\ell)).
\end{align}
For convenience, we set $\haf_\ell(Z)=1$ for 
$Z\in\CC^{J\times \dots\times J}$ 
with $J=\emptyset$. In the case $\ell=2$, \eqref{e762576876} reduces 
to \eqref{e276547}. See \citet[Section 6]{Barvinok1993} and
\citet[Section 5.1]{MR1943369}, for definitions of hyperhafnians  
using slightly different notation.  

\begin{theorem}\label{t7832577}
Let $\ell,m\in\NN$, $n=\ell m$, 
$Z=(z({j_1,\dots,j_\ell}))\in\CC^{\set{n}\times\dots\times\set{n}}$
be an $\ell$-dimensional symmetric matrix, and 
\begin{align*}
G_\ell(Z,k) 
=\frac{1}{\binomial{n}{\ell k}}\sum_{J\in\Subs{\set{n}}{\ell k}}
\ABS{\frac{k!(\ell!)^k}{(\ell k)!}\haf_\ell(Z[J,\dots,J])}^2
\quad \mbox{ for } k\in\{0,\dots,m\}.
\end{align*}
For $k\in\{0,\dots,m\}$, $d\in\NN$, and $w\in\Comp(k,d)$, we then have 
\begin{align*}
G_\ell(Z,k)
\leq\prod_{r=1}^d G_\ell(Z,w_r).
\end{align*} 
In particular, if $k=m$, then
\begin{align*}
\abs{\haf_\ell(Z)}
\leq\frac{n!}{m!(\ell!)^m}\prod_{r=1}^d\sqrt{G_\ell(Z,w_r)}.
\end{align*} 
\end{theorem}
For $\ell=1$, Theorem \ref{t7832577} reduces to Theorem \ref{t78368}. 
The proof can be found in the next section.

\section{Proofs of Theorems \ref{t45790}, \ref{t257876}, 
\ref{t7832577} and  Remark \ref{r2875876}} \label{s7265755}
The proofs of the theorems in Section \ref{s561476} are based on 
the following general theorem. 
\begin{theorem} \label{t26376566} 
Let $d,\ell\in\NN$ and $k,n\in\NN^\ell$ with $k_s\leq n_s$ for all 
$s\in\set{\ell}$. For $s\in\set{\ell}$, let
$A_s$ be a set with cardinality $\card{A_s}=n_s$ and
$w(s)=(w_{1,s},\dots,w_{d,s})\in\Comp(k_s,d)$. Further let 
$g_r:\,\bigtimes_{s=1}^\ell\Subs{A_s}{w_{r,s}}
\longrightarrow\CC$ be a map for all $r\in\set{d}$ and set 
$g=(g_1,\dots,g_d)$.
For $J_s\in\Subs{A_s}{k_s}$, $(s\in\set{\ell})$, let 
\begin{align}\label{e672587}
R(g,J_1,\dots,J_\ell)
&=\sum_{\newatop{(V_{1,1},\dots,V_{d,1})}{\in\Part(J_1,w(1))}}\dots
\sum_{\newatop{(V_{1,\ell},\dots,V_{d,\ell})}{\in
\Part(J_\ell,w(\ell))}}\prod_{r=1}^dg_r(V_{r,1},\dots,V_{r,\ell}).
\end{align}
Then 
\begin{align}
\begin{split}\label{e65463}
\lefteqn{\frac{1}{\prod_{s=1}^\ell\binomial{n_s}{k_s}}
\sum_{J_1\in\Subs{A_1}{k_1}}\dots\sum_{J_\ell\in\Subs{A_\ell}{k_\ell}}
\ABS{\Big(\prod_{s=1}^\ell \frac{w(s)!}{k_s!}\Big)
R(g,J_1,\dots,J_\ell)}^2}\\
&\leq \prod_{r=1}^d\Big(
\frac{1}{\prod_{s=1}^\ell\binomial{n_s}{w_{r,s}}}
\sum_{J_1\in\Subs{A_1}{w_{r,1}}}\dots
\sum_{J_\ell\in\Subs{A_\ell}{w_{r,\ell}}}
\abs{g_r(J_1,\dots,J_\ell)}^2\Big).
\end{split}
\end{align}
In particular, if $k=n$, then 
\begin{align*}
\ABS{R(g,A_1,\dots,A_\ell)} 
&\leq \frac{n!}{\prod_{s=1}^\ell w(s)!}
\prod_{r=1}^d\Big(\frac{1}{\prod_{s=1}^\ell\binomial{n_s}{w_{r,s}}}\\
&\quad{}\times
\sum_{J_1\in\Subs{A_1}{w_{r,1}}}\dots\sum_{J_\ell\in\Subs{A_\ell}{
w_{r,\ell}}}
\abs{g_r(J_1,\dots,J_\ell)}^2\Big)^{1/2}.
\end{align*}
\end{theorem}
The proof of Theorem \ref{t26376566} is given in Section \ref{s64514}. 
The case $\ell=1$ is discussed next. 
\begin{corollary}\label{c176547}
Let $d,k,n\in\NN$ with $k\leq n$, $A$ be a set with $\card{A}=n$, 
$w=(w_{1},\dots,w_{d})\in\Comp(k,d)$, 
$g_r:\,\Subs{A}{w_r}\longrightarrow\CC$ be a map for all $r\in\set{d}$,
and $g=(g_1,\dots,g_d)$. For $J\in\Subs{A}{k}$, let 
\begin{align*}
R(g,J)
=\sum_{\newatop{(V_{1},\dots,V_{d})}{\in\Part(J,w)}}
\prod_{r=1}^dg_r(V_r).
\end{align*}
Then 
\begin{align*} 
\frac{1}{\binomial{n}{k}}
\sum_{J\in\Subs{A}{k}}
\ABS{\frac{w!}{k!}R(g,J)}^2
&\leq \prod_{r=1}^d\Big(\frac{1}{\binomial{n}{w_r}}
\sum_{J\in\Subs{A}{w_r}}\abs{g_r(J)}^2\Big).
\end{align*}
In particular, if $k=n$, then 
\begin{align*}
\abs{R(g,A)} 
&\leq \frac{n!}{w!}\prod_{r=1}^d\Big(\frac{1}{\binomial{n}{w_r}}
\sum_{J\in\Subs{A}{w_r}}\abs{g_r(J)}^2\Big)^{1/2}.
\end{align*}
\end{corollary}


The following Lemmata \ref{l486369} and \ref{l324579} 
are proved in Section \ref{s2876563} and
contain generalized Laplace type expansions for multidimensional 
permanents and hyperhafnians, respectively. They indicate how 
to specify $g$ in Theorem \ref{t26376566} in order to prove 
the theorems of Section~\ref{s561476}. 

\begin{lemma} \label{l486369} 
Let $\ell\in\NN$, $J_1,\dots,J_\ell$, $K$ be sets with 
$\card{J_1}=\dots=\card{J_\ell}=\card{K}=k\in\NN$, 
$Z=(z(j_1,\dots,j_\ell,s))
\in\CC^{J_1\times\dots\times J_\ell\times K}$, 
$d\in\NN$, and $w\in\Comp(k,d)$.
For arbitrary $(W_{1},\dots,W_{d})\in\Part(K,w)$, we then
have
\begin{align}\label{e7821967}
\per_\ell(Z)
&=\sum_{\newatop{(V_{1,1},\dots,V_{d,1})}{\in\Part(J_1,w)}}\dots
\sum_{\newatop{(V_{1,\ell},\dots,V_{d,\ell})}{\in\Part(J_\ell,w)}}
\prod_{r=1}^d\per_\ell(Z[V_{r,1},\dots,V_{r,\ell},W_r]).
\end{align}
Furthermore 
\begin{align}\label{e7821968}
\per_\ell(Z)
&=\frac{w!}{k!}
\sum_{\newatop{(V_{1,1},\dots,V_{d,1})}{\in\Part(J_1,w)}}\dots
\sum_{\newatop{(V_{1,\ell},\dots,V_{d,\ell})}{\in\Part(J_\ell,w)}}
\sum_{\newatop{(W_{1},\dots,W_{d})}{\in\Part(K,w)}}
\prod_{r=1}^d\per_\ell(Z[V_{r,1},\dots,V_{r,\ell},W_{r}]).
\end{align}
\end{lemma}

In the case $\ell=1$, Lemma \ref{l486369} simplifies to the 
following corollary. 
\begin{corollary}  
Let  $J$ and $K$ be sets with $\card{J}=\card{K}=k\in\NN$, 
$Z=(z_{j,r})\in\CC^{J\times K}$, $d\in\NN$, and $w\in\Comp(k,d)$. 
For arbitrary $(W_1,\dots,W_d)\in\Part(K,w)$, we then have
\begin{align}\label{e7821965}
\per(Z)
&=\sum_{\newatop{(V_1,\dots,V_d)}{\in\Part(J,w)}}
\prod_{r=1}^d\per(Z[V_r,W_r]). 
\end{align}
Furthermore 
\begin{align*} 
\per(Z)
&=\frac{w!}{k!}\sum_{\newatop{(V_{1},\dots,V_{d})}{\in\Part(J,w)}}
\sum_{\newatop{(W_{1},\dots,W_{d})}{\in\Part(K,w)}}
\prod_{r=1}^d\per(Z[V_{r},W_{r}]).
\end{align*}
\end{corollary}

Identity \eqref{e7821965} can be found in 
\citet[(21.30) on page 559]{MR1440179}.
For $d=2$, it simplifies to the Laplace expansion for 
permanents given in \citet[Theorem 1.2, page 16]{MR504978}. 
\begin{lemma}\label{l324579}
Let $\ell,k\in\NN$, $J$ be a set with $\card{J}=\ell k$, 
$Z=(z({j_1,\dots,j_\ell}))\in\CC^{J\times\dots\times J}$
be a symmetric matrix, $d\in\NN$, and $w\in\Comp(k,d)$. Then we have
\begin{align*}
\haf_\ell(Z)
&=\frac{w!}{k!}
\sum_{\newatop{(V_1,\dots,V_d)}{\in\Part(J,\ell w)}}
\prod_{r=1}^d\haf_\ell(Z[V_r,\dots,V_r]).
\end{align*}
\end{lemma}
\begin{remark} 
Under the assumptions of Lemma \ref{l324579},
we get in the case $\ell=d=2$ and $w=(1,k-1)$ that 
\begin{align*}
\haf(Z)
&=\frac{1}{2k}\sum_{(j,k)\in J_{\neq}^2}z_{j,k}
\haf(Z[J\setminus \{j,k\},J\setminus \{j,k\}]),
\end{align*}
which also follows from the more general identity 
\begin{align*}
\haf(Z)
&=\sum_{k\in J\setminus\{j\}}z_{j,k}
\haf(Z[J\setminus \{j,k\},J\setminus \{j,k\}]) \quad \mbox{ for } 
j \in J;
\end{align*}
for instance, see \citet[(4.1.1.3) on page 94]{MR3558532}.
\end{remark}


\begin{proof}[Proof of Theorem \ref{t45790}]
Let $k\in\{0,\dots,m\}$, $K\in\Subs{M}{k}$, $d\in\NN$, 
$W\in\Part_d(K)$, and $w=(w_1,\dots,w_d)=\type(W)$. 
Then $w\in\Comp(k,d)$ and $W\in\Part(K,w)$. 
We let $g_r:\,(\Subs{\set{n}}{w_r})^\ell\longrightarrow\CC$ 
for $r\in\set{d}$, where
$g_r(V_{r,1},\dots,V_{r,\ell})=\frac{1}{(w_r!)^\ell}
\per_\ell(Z[V_{r,1},\dots,V_{r,\ell},W_r])$ for
$V_{r,1},\dots,V_{r,\ell}\in\Subs{\set{n}}{w_r}$.  
For $J_1,\dots,J_\ell\in\Subs{\set{n}}{k}$, we obtain from 
\eqref{e7821967} that
\begin{align}
\lefteqn{\frac{1}{(w!)^\ell}\per_\ell(Z[J_1,\dots,J_\ell,K])}
\nonumber\\
&=\frac{1}{(w!)^\ell}
\sum_{\newatop{(V_{1,1},\dots,V_{d,1})}{\in\Part(J_1,w)}}\dots
\sum_{\newatop{(V_{1,\ell},\dots,V_{d,\ell})}{\in\Part(J_\ell,w)}}
\prod_{r=1}^d\per_\ell(Z[V_{r,1},\dots,V_{r,\ell},W_r])\nonumber\\
&=\sum_{\newatop{(V_{1,1},\dots,V_{d,1})}{\in\Part(J_1,w)}}\dots
\sum_{\newatop{(V_{1,\ell},\dots,V_{d,\ell})}{\in\Part(J_\ell,w)}}
\prod_{r=1}^dg_r(V_{r,1},\dots,V_{r,\ell})\nonumber\\
&=:R(g,J_1,\dots,J_\ell),\label{e726345}
\end{align}
where $g=(g_1,\dots,g_d)$. By \eqref{e782566}, \eqref{e726345}, and 
Theorem \ref{t26376566}, we get  
\begin{align*}
f_\ell(Z,K) 
&=\frac{1}{\binomial{n}{k}^\ell}\sum_{J_1\in\Subs{\set{n}}{k}}
\dots\sum_{J_\ell\in\Subs{\set{n}}{k}}
\ABS{\frac{1}{(k!)^\ell}\per_\ell(Z[J_1,\dots,J_\ell,K])}^2\\
&=\frac{1}{\binomial{n}{k}^\ell}\sum_{J_1\in\Subs{\set{n}}{k}}
\dots\sum_{J_\ell\in\Subs{\set{n}}{k}}
\ABS{\Big(\frac{w!}{k!}\Big)^\ell R(g,J_1,\dots,J_\ell)}^2\\
&\leq \prod_{r=1}^d\Big(\frac{1}{\binomial{n}{w_r}^\ell}
\sum_{J_1\in\Subs{\set{n}}{w_{r}}}\dots
\sum_{J_\ell\in\Subs{\set{n}}{w_{r}}}\abs{g_r(J_1,\dots,J_\ell)}^2
\Big)\\
&=\prod_{r=1}^d\Big(\frac{1}{\binomial{n}{w_r}^\ell}
\sum_{J_1\in\Subs{\set{n}}{w_{r}}}\dots
\sum_{J_\ell\in\Subs{\set{n}}{w_{r}}}
\ABS{\frac{1}{(w_r!)^\ell}\per_\ell(Z[J_1,\dots,J_\ell,W_r])}^2\Big)\\
&=\prod_{r=1}^d f_\ell(Z,W_r).\qedhere
\end{align*}
\end{proof}


\begin{proof}[Proof of Remark \ref{r2875876}]
Let the assumptions of Theorem \ref{t45790} hold. 
If condition \ref{r2875876.i} of Remark \ref{r2875876} is true, then 
both sides of \eqref{e84378} are equal to zero. Now suppose that 
condition \ref{r2875876.ii} is true. Let 
$k\in\set{m}$, $K\in\Subs{M}{k}$, $d\in\NN$, $W\in\Part_d(K)$, and
$w=\type(W)\in\Comp(k,d)$. Using \eqref{e7821967}, we obtain for 
$J_1,\dots,J_\ell\in\Subs{\set{n}}{k}$ that 
\begin{align*}
\per_\ell(Z[J_1,\dots,J_\ell,K])
&=\sum_{\newatop{(V_{1,1},\dots,V_{d,1})}{\in\Part(J_1,w)}}\dots
\sum_{\newatop{(V_{1,\ell},\dots,V_{d,\ell})}{\in\Part(J_\ell,w)}}
\prod_{r=1}^dy_r 
=\Big(\frac{k!}{w!}\Big)^\ell \prod_{r=1}^dy_r. 
\end{align*}
From \eqref{e782566}, we then get 
\begin{align*}
f_\ell(Z,K)
&=\prod_{r=1}^d\frac{\abs{y_r}^2}{(w_r!)^{2\ell}}
=\prod_{r=1}^d\Big(\frac{1}{\binomial{n}{w_r}^\ell}
\sum_{J_1\in\Subs{\set{n}}{w_r}}\dots
\sum_{J_\ell\in\Subs{\set{n}}{w_r}}
\ABS{\frac{y_r}{(w_r!)^\ell}}^2\Big)\\
&=\prod_{r=1}^d f_\ell(Z,W_r). \qedhere
\end{align*}
\end{proof}

\begin{proof}[Proof of Theorem \ref{t257876}]
Let $g_r:\,(\Subs{\set{n}}{w_r})^{\ell}\times\Subs{M}{w_r}
\longrightarrow\CC$ with  
$g_r(V_{r,1},\dots,V_{r,\ell+1})
=\frac{1}{(w_r!)^\ell}\per_\ell(Z[V_{r,1},\dots,V_{r,\ell+1}])$ 
for $r\in\set{d}$, $V_{r,1},\dots,V_{r,\ell}\in\Subs{\set{n}}{w_r}$, 
and $V_{r,\ell+1}\in\Subs{M}{w_r}$. For 
$J_1,\dots,J_\ell\in\Subs{\set{n}}{k}$ and $J_{\ell+1}\in\Subs{M}{k}$, 
\eqref{e7821968} implies that 
\begin{align}
\lefteqn{\frac{k!}{(w!)^{\ell+1}}
\per_\ell(Z[J_1,\dots,J_{\ell+1}])}\nonumber\\
&=\frac{1}{(w!)^\ell}
\sum_{\newatop{(V_{1,1},\dots,V_{d,1})}{\in\Part(J_1,w)}}\dots
\sum_{\newatop{(V_{1,\ell+1},\dots,V_{d,\ell+1})}{
\in\Part(J_{\ell+1},w)}}
\prod_{r=1}^d\per_\ell(Z[V_{r,1},\dots,V_{r,\ell+1}])\nonumber\\
&=\sum_{\newatop{(V_{1,1},\dots,V_{d,1})}{\in\Part(J_1,w)}}\dots
\sum_{\newatop{(V_{1,\ell+1},\dots,V_{d,\ell+1})}{
\in\Part(J_{\ell+1},w)}}
\prod_{r=1}^dg_r(V_{r,1},\dots,V_{r,\ell+1})
=: R(g,J_1,\dots,J_{\ell+1}),\label{e7870}
\end{align}
where $g=(g_1,\dots,g_d)$. Then, \eqref{t257877}, \eqref{e7870}, and 
Theorem \ref{t26376566} give
\begin{align*}
F_\ell(Z,k)
&=\frac{1}{\binomial{n}{k}^\ell\binomial{m}{k}}
\sum_{J_1\in\Subs{\set{n}}{k}}
\dots\sum_{J_\ell\in\Subs{\set{n}}{k}}
\sum_{J_{\ell+1}\in\Subs{M}{k}}
\ABS{\frac{1}{(k!)^\ell}\per_\ell(Z[J_1,\dots,J_{\ell+1}])}^2\\
&=\frac{1}{\binomial{n}{k}^\ell\binomial{m}{k}}
\sum_{J_1\in\Subs{\set{n}}{k}}\dots\sum_{J_\ell\in\Subs{\set{n}}{k}}
\sum_{J_{\ell+1}\in\Subs{M}{k}}
\ABS{
\Big(\frac{w!}{k!}\Big)^{\ell+1}R(g,J_1,\dots,J_{\ell+1})}^2\\
&\leq \prod_{r=1}^d
\Big(\frac{1}{\binomial{n}{w_r}^\ell\binomial{m}{w_r}}
\sum_{J_1\in\Subs{\set{n}}{w_r}}\dots
\sum_{J_\ell\in\Subs{\set{n}}{w_r}}\sum_{J_{\ell+1}\in\Subs{M}{w_r}}
\abs{g_r(J_1,\dots,J_{\ell+1})}^2\Big)\\
&=\prod_{r=1}^dF_\ell(Z,w_r).\qedhere
\end{align*}
\end{proof}

\begin{proof}[Proof of Theorem \ref{t7832577}]
Let $k\in\{0,\dots,m\}$, $d\in\NN$, and $w\in\Comp(k,d)$. 
Further, let $g_r:\,\Subs{\set{n}}{\ell w_r}\longrightarrow\CC$ with
$g_r(V_r)=\frac{w_r!(\ell!)^{w_r}}{(\ell w_r)!}
\haf_\ell(Z[V_r,\dots,V_r])$ for 
$r\in\set{d}$, $V_r\in\Subs{\set{n}}{\ell w_r}$. 
For $J\in\Subs{\set{n}}{\ell k}$, we get from Lemma \ref{l324579} that
\begin{align}
\frac{k!(\ell!)^k}{(\ell w)!}\haf_\ell(Z[J,\dots,J])
&=\frac{(\ell!)^kw!}{(\ell w)!}
\sum_{\newatop{(V_1,\dots,V_d)}{\in\Part(J,\ell w)}}
\prod_{r=1}^d\haf_\ell(Z[V_r,\dots,V_r])\nonumber\\
&=\sum_{\newatop{(V_1,\dots,V_d)}{\in\Part(J,\ell w)}}
\prod_{r=1}^dg_r(V_r)
=:R(g,J),\label{e7625}
\end{align}
where $g=(g_1,\dots,g_d)$. Therefore, \eqref{e7625} and 
Corollary \ref{c176547} give
\begin{align*}
G_\ell(Z,k)
&=\frac{1}{\binomial{n}{\ell k}}\sum_{J\in\Subs{\set{n}}{\ell k}}
\ABS{\frac{(\ell w)!}{(\ell k)!}R(g,J)}^2\\
&\leq \prod_{r=1}^d\Big(\frac{1}{\binomial{n}{\ell w_r}}
\sum_{J\in\Subs{\set{n}}{\ell w_r}}\abs{g_r(J)}^2\Big)
=\prod_{r=1}^dG_\ell(Z,w_r). \qedhere
\end{align*} 
\end{proof}

\section{An auxiliary inequality and proof of Theorem \ref{t26376566}}
\label{s64514}

The proof of  Theorem \ref{t26376566} requires the following 
result. 
\begin{theorem}\label{t983467}
Let $n\in\NN$,  $A$ be a set with $\card{A}=n$, 
$j,k\in\Zpl$ with $j\leq k\leq n$, and
$g:\,\Subs{A}{j}\longrightarrow[0,\infty)$ and 
$h:\,\Subs{A}{k-j}\longrightarrow[0,\infty)$ be two maps. 
For $J\in\Subs{A}{k}$, let 
\begin{align*}
p(J)
&=(g*_j h)(J)
=\sum_{I\in\Subs{J}{j}}g(I)h(J\setminus I).
\end{align*}
Then 
\begin{align}\label{e717659}
\frac{1}{\binomial{n}{k}}
\sum_{J\in\Subs{A}{k}}
\Big(\frac{p(J)}{\binomial{k}{j}}\Big)^2
&\leq \Big(\frac{1}{\binomial{n}{j}}
\sum_{I\in\Subs{A}{j}}g(I)^2
\Big)\Big(\frac{1}{\binomial{n}{k-j}}
\sum_{J\in\Subs{A}{k-j}} h(J)^2\Big).
\end{align}
\end{theorem}
For the proofs of the theorems given in this section and the one of
Remark \ref{r28759} below, see Section \ref{s1534215}. 
 
Let the assumptions of Theorem \ref{t983467} hold. 
Suppose that the functions $g$ and $h$ are extended to the power set 
$2^A$ of $A$ such that $g(I)=0$ for all $I\subseteq A$ with 
$\card{I}\neq j$ and $h(J)=0$ for all $J\subseteq A$ with 
$\card{J}\neq k-j$. For $J\in\Subs{A}{k}$, we then have 
$(g*_j h)(J)=\sum_{I\subseteq J}g(I)h(J\setminus I)$, 
where the latter expression is called the subset convolution of $g$
and $h$, e.g.\ see \citet{MR2402429}. 

Theorem \ref{t983467} and Remark \ref{r28759} below are non-trivial
generalizations of \citet[Proposition 3.1 and Remark 3.1]{MR3916882}.
In fact, in \cite{MR3916882}, it was assumed that there are numbers 
$g_i\in[0,\infty)$ for $i\in A$ such that $g(I)=\prod_{i\in I}g_i$
for all $I\in\Subs{A}{j}$. 
We have not been able to generalize the method of proof of
\cite[Proposition 3.1]{MR3916882} to give a proof of 
Theorem \ref{t983467}.  

\begin{remark}\label{r28759} 
Let the assumptions of Theorem \ref{t983467} hold. 
In \eqref{e717659}, equality holds if and only if at least one of the 
following five conditions is valid: 
\begin{enumerate}[(i)]

\item \label{r28759.i} $j\in\{0, k\}$ or 

\item  $g(I)=0$ for all $I\in\Subs{A}{j}$ or 

\item  $h(J)=0$ for all $J\in\Subs{A}{k-j}$ or  

\item \label{r28759.iv} $k=n$ and a number $x\in[0,\infty)$ exists 
such that $g(I)=xh(A\setminus I)$ for all 
$I\in\Subs{A}{j}$ or 

\item \label{r28759.v} $g(I)=g(I')$ for all $I,I'\in\Subs{A}{j}$
and $h(J)=h(J')$ for all $J,J'\in\Subs{A}{k-j}$. 

\end{enumerate}
\end{remark}

The following result is a consequence of Theorem \ref{t983467}.
But it also reduces to this theorem for $\ell=1$. 

\begin{theorem} \label{c72566}
Let $\ell\in\NN$, $n\in\NN^\ell$, 
$A_s$ be a set with $\card{A_s}=n_s$ for all $s\in\set{\ell}$,
$j,k\in\Zpl^\ell$ with $j_s\leq k_s\leq n_s$ for all 
$s\in\set{\ell}$, and 
$g:\,\bigtimes_{s=1}^\ell\Subs{A_s}{j_s}\longrightarrow[0,\infty)$ and
$h:\,\bigtimes_{s=1}^\ell\Subs{A_s}{k_s-j_s}
\longrightarrow[0,\infty)$ be two maps. 
For $J_s\in\Subs{A_s}{k_s}$, $(s\in\set{\ell})$,  let  
\begin{align*}
p(J_1,\dots,J_\ell)
&=(g*_j h)(J_1,\dots,J_\ell)
=\sum_{I_1\in\Subs{J_1}{j_1}}\dots\sum_{I_\ell\in\Subs{J_\ell}{j_\ell}}
g(I_1,\dots,I_\ell)h(J_1\setminus I_1,\dots,J_\ell\setminus I_\ell).
\end{align*}
Then 
\begin{align}\label{e78369888}
\begin{split}
\lefteqn{\frac{1}{\prod_{s=1}^\ell\binomial{n_s}{k_s}}
\sum_{J_1\in\Subs{A_1}{k_1}}
\dots\sum_{J_\ell\in\Subs{A_\ell}{k_\ell}}\Big(
\frac{p(J_1,\dots,J_\ell)}{
\prod_{s=1}^\ell\binomial{k_s}{j_s}}\Big)^2}\\
&\leq \Big(\frac{1}{\prod_{s=1}^\ell\binomial{n_s}{j_s}}
\sum_{I_1\in\Subs{A_1}{j_1}}\dots\sum_{I_\ell\in\Subs{A_\ell}{j_\ell}}
g(I_1,\dots,I_\ell)^2\Big)\\
&\quad{}\times
\Big(\frac{1}{\prod_{s=1}^\ell\binomial{n_s}{k_s-j_s}}
\sum_{J_1\in\Subs{A_1}{k_1-j_1}}
\dots\sum_{J_\ell\in\Subs{A_\ell}{k_\ell-j_\ell}}
h(J_1,\dots,J_\ell)^2\Big).
\end{split}
\end{align}
\end{theorem}

\begin{proof}[Proof of Theorem \ref{t26376566}]
Let $T_d$ denote the left-hand side of the inequality in 
\eqref{e65463}. We now use induction over $d$ to show that 
\begin{align*}
T_d
&\leq \prod_{r=1}^d\Big(
\frac{1}{\prod_{s=1}^\ell\binomial{n_s}{w_{r,s}}}
\sum_{J_1\in\Subs{A_1}{w_{r,1}}}\dots
\sum_{J_\ell\in\Subs{A_\ell}{w_{r,\ell}}}
\abs{g_r(J_1,\dots,J_\ell)}^2\Big).
\end{align*}
For $d=1$, we have $w(s)=w_{1,s}=k_s$ 
for all $s\in\set{\ell}$, 
$R(g,J_1,\dots,J_\ell)=g_1(J_1,\dots,J_\ell)$ for 
$J_s\in\Subs{A_s}{k_s}$, $(s\in\set{\ell})$, and hence 
\begin{align*}
T_1
&=\frac{1}{\prod_{s=1}^\ell\binomial{n_s}{w_{1,s}}}
\sum_{J_1\in\Subs{A_1}{w_{1,1}}}\dots
\sum_{J_\ell\in\Subs{A_\ell}{w_{1,\ell}}} 
\abs{g_1(J_1,\dots,J_\ell)}^2.
\end{align*}
In the proof of the assertion for $d\in\NN\setminus\set{1}$,
we assume its validity for $d-1$. From \eqref{e672587}, we obtain  
that, for $J_s\in\Subs{A_s}{k_s}$, $(s\in\set{\ell})$, 
\begin{align} 
\lefteqn{\Big(\prod_{s=1}^\ell \frac{w(s)!}{k_s!}\Big)
R(g,J_1,\dots,J_\ell)}\nonumber\\
&=\frac{1}{\prod_{s=1}^\ell\binomial{k_s}{w_{d,s}}}
\sum_{V_{d,1}\in\Subs{J_1}{w_{d,1}}}\dots
\sum_{V_{d,\ell}\in\Subs{J_\ell}{w_{d,\ell}}} 
g_d(V_{d,1},\dots,V_{d,\ell})\nonumber\\
&\quad{}\times
\Big(\prod_{s=1}^\ell
\frac{\widetilde{w}(s)!}{(k_s-w_{d,s})!}\Big)
\sum_{\newatop{(V_{1,1},\dots,V_{d-1,1})}{
\in\Part(J_1\setminus V_{d,1},\widetilde{w}(1))}}
\dots\sum_{\newatop{(V_{1,\ell},\dots,V_{d-1,\ell})}{
\in\Part(J_\ell\setminus V_{d,\ell},\widetilde{w}(\ell))}}
\prod_{r=1}^{d-1}g_r(V_{r,1},\dots,V_{r,\ell}),\label{e987576}
\end{align}
where $\widetilde{w}(s)=(w_{1,s},\dots,w_{d-1,s})
\in\Comp(k_s-w_{d,s},d-1)$ for all $s\in\set{\ell}$. 
Let $\widetilde{g}=(g_1,\dots,g_{d-1})$ and 
\begin{align}
h(J_1,\dots,J_\ell)
&=\Big(\prod_{s=1}^\ell
\frac{\widetilde{w}(s)!}{(k_s-w_{d,s})!}\Big)
\widetilde{R}(\widetilde{g},J_1,\dots,J_\ell),\label{e652147}\\
\widetilde{R}(\widetilde{g},J_1,\dots,J_\ell)
&=\sum_{\newatop{(V_{1,1},\dots,V_{d-1,1})}{
\in\Part(J_1,\widetilde{w}(1))}}
\dots\sum_{\newatop{(V_{1,\ell},\dots,V_{d-1,\ell})}{
\in\Part(J_\ell,\widetilde{w}(\ell))}}
\prod_{r=1}^{d-1}g_r(V_{r,1},\dots,V_{r,\ell})\label{e652148}
\end{align}
for $J_s\in\Subs{A_s}{k_s-w_{d,s}}$. By \eqref{e987576}, we get
\begin{align*}
\lefteqn{\Big(\prod_{s=1}^\ell \frac{w(s)!}{k_s!}\Big)
R(g,J_1,\dots,J_\ell)}\nonumber\\
&=\frac{1}{\prod_{s=1}^\ell\binomial{k_s}{w_{d,s}}}
\sum_{I_1\in\Subs{J_1}{w_{d,1}}}\dots
\sum_{I_\ell\in\Subs{J_\ell}{w_{d,\ell}}} g_d(I_1,\dots,I_\ell)
h(J_1\setminus I_1,\dots,J_\ell\setminus I_\ell),
\end{align*}
which together with Theorem \ref{c72566} implies
\begin{align*}
T_d
&=\frac{1}{\prod_{s=1}^\ell\binomial{n_s}{k_s}}
\sum_{J_1\in\Subs{A_1}{k_1}}\dots\sum_{J_\ell\in\Subs{A_\ell}{k_\ell}}
\ABS{\Big(\prod_{s=1}^\ell \frac{w(s)!}{k_s!}\Big)
R(g,J_1,\dots,J_\ell)}^2\\
&=\frac{1}{\prod_{s=1}^\ell\binomial{n_s}{k_s}}
\sum_{J_1\in\Subs{A_1}{k_1}}\dots\sum_{J_\ell\in\Subs{A_\ell}{k_\ell}}
\ABS{\frac{1}{\prod_{s=1}^\ell\binomial{k_s}{w_{d,s}}}\\
&\quad{}\times
\sum_{I_1\in\Subs{J_1}{w_{d,1}}}\dots
\sum_{I_\ell\in\Subs{J_\ell}{w_{d,\ell}}} g_d(I_1,\dots,I_\ell)
h(J_1\setminus I_1,\dots,J_\ell\setminus I_\ell)}^2\\
&\leq \Big(\frac{1}{\prod_{s=1}^\ell\binomial{n_s}{w_{d,s}}}
\sum_{I_1\in\Subs{A_1}{w_{d,1}}}
\dots\sum_{I_\ell\in\Subs{A_\ell}{w_{d,\ell}}}
\abs{g_d(I_1,\dots,I_\ell)}^2\Big)\\
&\quad{}\times
\Big(\frac{1}{\prod_{s=1}^\ell\binomial{n_s}{k_s-w_{d,s}}}
\sum_{J_1\in\Subs{A_1}{k_1-w_{d,1}}}
\dots\sum_{J_\ell\in\Subs{A_\ell}{k_\ell-w_{d,\ell}}}
\abs{h(J_1,\dots,J_\ell)}^2\Big).
\end{align*}
In view of \eqref{e652147} and \eqref{e652148}, we see that 
the assertion now follows by applying the induction hypothesis
to the second factor above.
\end{proof} 

\section{Remaining proofs} \label{s82576}
\subsection{Proofs of Lemmata \ref{l486369} and \ref{l324579}}
\label{s2876563}
The following remark contains some arguments needed in the proofs 
below. 
\begin{remark}\label{r28659}
Let $J,K$ be non-empty sets with $\card{J}=\card{K}=k$.
\begin{enumerate}\itemsep-2pt

\item \label{r28659.a}
Let $j\in J_{\neq}^K$, $W\subseteq K$, and $M$ be a set with 
$\card{M}=\card{W}$. For every $t\in(j[W])_{\neq}^M$, there is 
exactly one $t'\in W_{\neq}^M$ such that $t=j|_W\circ t'$, and vice 
versa. Here and henceforth, $j[W]=\{j_s\,|\,s\in W\}$ is the image of 
$W$ under the map $j$ and $j|_W:\,W\longrightarrow j[W]$ is the 
restriction of $j$ to $W$. In particular, we have 
$(j[W])_{\neq}^M=\{j|_W\circ t'\;|\;t'\in W_{\neq}^M\}$. 

\item \label{r28659.b} Let $t\in K_{\neq}^K$. 
For every $j\in J_{\neq}^K$, there is exactly one $j'\in J_{\neq}^K$
such that $j=j'\circ t$, and vice versa. In particular, we have 
$J_{\neq}^K=\{j'\circ t\,|\,j'\in J_{\neq}^K\}$.

\item \label{r28659.c}
Let $d\in\NN$, $w\in\Comp(k,d)$, $(W_1,\dots,W_d)\in\Part(K,w)$, 
and $V_1,\dots,V_d\in2^J$. Then we have  
$V=(V_1,\dots,V_d)\in\Part(J,w)$ if and only if 
a map $j\in J_{\neq}^K$ exists such that $V_r=j[W_r]$ for all 
$r\in\set{d}$. Clearly, in this case there are $w!$ of such 
maps~$j$. Consequently,  
$\Part(J,w)=\{(j[W_1],\dots,j[W_d])\,|\,j\in J_{\neq}^K\}$ and 
every sum over $(V_1,\dots,V_d)\in\Part(J,w)$ can be written as a 
sum over $j\in J_{\neq}^K$ divided by $w!$, where the sets 
$V_1,\dots,V_d$ have to be replaced with $j[W_1],\dots,j[W_d]$, 
respectively.

\end{enumerate}
\end{remark}


\begin{proof}[Proof of Lemma \ref{l486369}]
For all $j^{(1)}\in (J_1)_{\neq}^K,\dots,
j^{(\ell)}\in (J_\ell)_{\neq}^K$ and $r\in\set{d}$, \eqref{e8612766}
and Remark \ref{r28659}(\ref{r28659.a}) give 
\begin{align}
\per_\ell(Z[j^{(1)}[W_r],\dots,j^{(\ell)}[W_r],W_r])
&=\sum_{t^{(1)}\in (j^{(1)}[W_r])_{\neq}^{W_r}}\dots
\sum_{t^{(\ell)}\in (j^{(\ell)}[W_r])^{W_r}_{\neq}}
\prod_{s\in W_r}z(t_s^{(1)},\dots,t_s^{(\ell)},s)\nonumber\\
&=\sum_{t^{(1)},\dots,t^{(\ell)}\in (W_r)_{\neq}^{W_r}}
\prod_{s\in W_r}z(j^{(1)}(t_s^{(1)}),\dots,
j^{(\ell)}(t_s^{(\ell)}),s).\label{e9727799}
\end{align}
For arbitrary $t^{(r,1)},\dots,t^{(r,\ell)}\in (W_r)_{\neq}^{W_r}$, 
$(r\in\set{d})$, we obtain from \eqref{e8612766} and 
Remark \ref{r28659}(\ref{r28659.b}) that
\begin{align}
\per_\ell(Z)
&=\sum_{j^{(1)}\in (J_1)_{\neq}^K}\dots
\sum_{j^{(\ell)}\in (J_\ell)^K_{\neq}}
\prod_{s\in K}z(j_s^{(1)},\dots,j_s^{(\ell)},s)\nonumber\\
&=\sum_{j^{(1)}\in (J_1)_{\neq}^K}\dots\sum_{j^{(\ell)}\in 
(J_\ell)^K_{\neq}}\prod_{r=1}^d
\Big(\prod_{s\in W_r}z(j_s^{(1)},\dots,j_s^{(\ell)},s)\Big)\nonumber\\
&=\sum_{j^{(1)}\in (J_1)_{\neq}^{K}}\dots
\sum_{j^{(\ell)}\in (J_\ell)_{\neq}^{K}} 
\prod_{r=1}^d\Big(\prod_{s\in W_r}z(j^{(1)}(t_s^{(r,1)}),\dots,
j^{(\ell)}(t_s^{(r,\ell)}),s)\Big).\label{e15790}
\end{align}
Using Remark \ref{r28659}(\ref{r28659.c}), \eqref{e9727799}, and 
\eqref{e15790}, we get
\begin{align*}
\lefteqn{\sum_{\newatop{(V_{1,1},\dots,V_{d,1})}{\in\Part(J_1,w)}}\dots
\sum_{\newatop{(V_{1,\ell},\dots,V_{d,\ell})}{\in\Part(J_\ell,w)}}
\prod_{r=1}^d\per_\ell(Z[V_{r,1},\dots,V_{r,\ell},W_r])}\\
&=\frac{1}{(w!)^\ell}\sum_{j^{(1)}\in (J_1)_{\neq}^{K}}\dots
\sum_{j^{(\ell)}\in (J_\ell)_{\neq}^{K}} 
\prod_{r=1}^d\per_\ell(Z[j^{(1)}[W_r],\dots,j^{(\ell)}[W_r],W_r])\\
&=\frac{1}{(w!)^\ell}\sum_{j^{(1)}\in (J_1)_{\neq}^{K}}\dots
\sum_{j^{(\ell)}\in (J_\ell)_{\neq}^{K}}\prod_{r=1}^d\Big(
\sum_{t^{(r,1)},\dots,t^{(r,\ell)}\in (W_r)_{\neq}^{W_r}}
\prod_{s\in W_r}z(j^{(1)}(t_s^{(r,1)}),\dots,
j^{(\ell)}(t_s^{(r,\ell)}),s)\Big)\\
&=\frac{1}{(w!)^\ell}
\sum_{t^{(1,1)},\dots,t^{(1,\ell)}\in (W_1)_{\neq}^{W_1}}
\dots\sum_{t^{(d,1)},\dots,t^{(d,\ell)}\in (W_d)_{\neq}^{W_d}}\\
&\quad{}\times
\sum_{j^{(1)}\in (J_1)_{\neq}^{K}}\dots
\sum_{j^{(\ell)}\in (J_\ell)_{\neq}^{K}}
\prod_{r=1}^d\Big(\prod_{s\in W_r}z(j^{(1)}(t_s^{(r,1)}),\dots,
j^{(\ell)}(t_s^{(r,\ell)}),s)\Big)\\
&=\per_\ell(Z),
\end{align*}
which shows \eqref{e7821967}. The left-hand side in \eqref{e7821967}
does not depend on $(W_1,\dots,W_d)$. 
Summing up over $(W_1,\dots,W_d)\in\Part(K,w)$ and dividing 
by $\card{\Part(K,w)}=\frac{k!}{w!}$, \eqref{e7821968} is shown. 
\end{proof}


\begin{proof}[Proof of Lemma \ref{l324579}]
Let $n=\ell k$ and $(W_1,\dots,W_d)\in\Part(\set{n},\ell w)$.  
It suffices to show that 
\begin{align}\label{e514232}
H_\ell(Z)=\sum_{j\in J_{\neq}^{n}}\prod_{r=1}^d
\Big(\frac{1}{(\ell w_r)!}H_\ell(Z[j[W_r],\dots,j[W_r]])\Big),
\end{align}
where $H_\ell(Z)=k!(\ell!)^k\haf_\ell(Z)$. 
In fact, this together with Remark \ref{r28659}(\ref{r28659.c}) 
implies that 
\begin{align*}
\haf_\ell(Z)
&=\frac{1}{k!(\ell!)^k}H_\ell(Z)
=\frac{(\ell w)!}{k!(\ell!)^k}
\sum_{\newatop{(V_1,\dots,V_d)}{\in\Part(J,\ell w)}}\prod_{r=1}^d
\Big(\frac{1}{(\ell w_r)!}H_\ell(Z[V_r,\dots,V_r])\Big)\\
&=\frac{(\ell w)!}{k!(\ell!)^k}
\sum_{\newatop{(V_1,\dots,V_d)}{\in\Part(J,\ell w)}}\prod_{r=1}^d
\Big(\frac{w_r!(\ell!)^{w_r}}{(\ell w_r)!}\haf_\ell(Z[V_r,
\dots,V_r])\Big)\\
&=\frac{w!}{k!}\sum_{\newatop{(V_1,\dots,V_d)}{\in\Part(J,\ell w)}}
\prod_{r=1}^d\haf_\ell(Z[V_r,\dots,V_r]).
\end{align*}
Let $M=(M_1,\dots,M_d)\in\Part(\set{k},w)$ and 
\begin{align*}
\widetilde{M}_r
=\{\ell(s-1)+i\,|\,s\in M_r,i\in\set{\ell}\} \quad \mbox{for } 
r\in\set{d}.
\end{align*}
It is clear that 
$(\widetilde{M}_1,\dots,\widetilde{M}_d)\in\Part(\set{n},\ell w)$. 
If $t^{(r)}\in (W_r)_{\neq}^{\widetilde{M}_r}$ for all 
$r\in\set{d}$ and $t\in\set{n}_{\neq}^{n}$ with 
\begin{align*}
t(\ell(s-1)+1)=t^{(r)}(\ell(s-1)+1),\dots,t(\ell(s-1)+\ell)
=t^{(r)}(\ell(s-1)+\ell)
\end{align*}
for all $r\in\set{d}$ and $s\in M_r$, then Remark 
\ref{r28659}(\ref{r28659.b}) gives
\begin{align}
H_\ell(Z)
&=\sum_{j\in J_{\neq}^{n}}\prod_{r=0}^{k-1} 
z(j(r\ell+1),\dots,j(r\ell+\ell))\nonumber\\
&=\sum_{j\in J_{\neq}^{n}}\prod_{r=1}^{k} 
z(j(t(\ell(r-1)+1)),\dots,j(t(\ell(r-1)+\ell)))\nonumber\\
&=\sum_{j\in J_{\neq}^{n}}\prod_{r=1}^d
\Big(\prod_{s\in M_r}z(j(t(\ell(s-1)+1)),\dots,j(t(\ell(s-1)+\ell)))
\Big)\nonumber\\
&=\sum_{j\in J_{\neq}^{n}}\prod_{r=1}^d\Big(
\prod_{s\in M_r} z(j(t^{(r)}(\ell(s-1)+1)),
\dots,j(t^{(r)}(\ell(s-1)+\ell)))\Big). \label{e7819858}
\end{align}
For $r\in\set{d}$ and  $j\in J_{\neq}^{n}$, 
we get from Remark \ref{r28659}(\ref{r28659.a}) that
\begin{align}
H_\ell(Z[j[W_r],\dots,j[W_r]])
&=w_r!(\ell!)^{w_r}\haf_\ell(Z[j[W_r],\dots,j[W_r]])\nonumber\\
&=\sum_{t\in (j[W_r])_{\neq}^{\ell w_r}}\prod_{s=1}^{w_r} 
z(t(\ell(s-1)+1),\dots,t(\ell(s-1)+\ell))\nonumber\\
&=\sum_{t\in (W_r)_{\neq}^{\ell w_r}}\prod_{s=1}^{w_r} 
z(j(t(\ell(s-1)+1)),\dots,j(t(\ell(s-1)+\ell)))\nonumber\\
&=\sum_{t\in (W_r)_{\neq}^{\widetilde{M}_r}}
\prod_{s\in M_r} z(j(t(\ell(s-1)+1)),\dots,j(t(\ell(s-1)+\ell))).
\label{e618759}
\end{align}
Therefore, using \eqref{e618759} and \eqref{e7819858}, we obtain
\begin{align*}
\lefteqn{\sum_{j\in J_{\neq}^{n}}\prod_{r=1}^d
\Big(\frac{1}{(\ell w_r)!}H_\ell(Z[j[W_r],\dots,j[W_r]])\Big)}\\
&=\sum_{j\in J_{\neq}^{n}}\prod_{r=1}^d
\Big(\frac{1}{(\ell w_r)!}\sum_{t\in (W_r)_{\neq}^{\widetilde{M}_r}}
\prod_{s\in M_r} z(j(t(\ell(s-1)+1)),\dots,j(t(\ell(s-1)+\ell)))
\Big)\\
&=\frac{1}{(\ell w)!}\sum_{t^{(1)}\in (W_1)_{\neq}^{\widetilde{M}_1}}
\dots\sum_{t^{(d)}\in (W_d)_{\neq}^{\widetilde{M}_d}}\\
&\quad{}\times \sum_{j\in J_{\neq}^{n}}\prod_{r=1}^d\Big(
\prod_{s\in M_r} z(j(t^{(r)}(\ell(s-1)+1)),\dots,
j(t^{(r)}(\ell(s-1)+\ell)))\Big)\\
&=H_\ell(Z),
\end{align*}
which implies \eqref{e514232}. \qedhere
\end{proof}

\subsection{Proofs of Theorems \ref{t983467}, \ref{c72566}, and
Remark \ref{r28759}} \label{s1534215}

\begin{proof}[Proof of Theorem \ref{t983467}]
Let us first mention some simple facts. 
If $j=0$ or $j= k$, then in \eqref{e717659} equality holds. 
Further, in the case $k=n$, \eqref{e717659}  
easily follows from the Cauchy-Schwarz inequality.
For the proof of \eqref{e717659}, we use induction over $n$. 
The observations above imply the validity of the assertion
in the case $n\in\set{2}$. 
In the proof of the assertion for general $n\in\NN\setminus\set{2}$, 
we assume its validity for $n-1$. We may assume that 
$0<j< k<n$.  Let $m\in A$ be fixed and set $A'=A\setminus\{m\}$. 
For $J\in\Subs{A'}{k-1}$, let 
\begin{align*} 
p_1(J)=\sum_{I\in\Subs{J}{j-1}}g(I\cup\{m\})h(J\setminus I),\quad 
p_2(J)=\sum_{I\in\Subs{J}{j}}g(I)h((J\setminus I)\cup\{m\}). 
\end{align*}
Then we have 
\begin{align*}
p(J\cup\{m\})
&=\sum_{I\in\Subs{J\cup\{m\}}{j}}g(I)
h((J\cup\{m\})\setminus I)\\
&=\sum_{I\in\Subs{J}{j-1}}g(I\cup\{m\})h(J\setminus I)
+\sum_{I\in\Subs{J}{j}}g(I)h((J\cup\{m\})\setminus I)
=p_1(J)+p_2(J)
\end{align*}
for all $J\in\Subs{A'}{k-1}$. Using the Minkowski inequality, we obtain
\begin{align*}
T&:=\sum_{J\in\Subs{A}{k}}p(J)^2
=\sum_{J\in\Subs{A'}{k}}p(J)^2
+\sum_{J\in\Subs{A'}{k-1}}p(J\cup\{m\})^2\\
&=\sum_{J\in\Subs{A'}{k}}p(J)^2
+\sum_{J\in\Subs{A'}{k-1}}(p_1(J)+p_2(J))^2\\
&\leq \sum_{J\in\Subs{A'}{k}}p(J)^2
+\Big(\Big(\sum_{J\in\Subs{A'}{k-1}}p_1(J)^2\Big)^{1/2}
+\Big(\sum_{J\in\Subs{A'}{k-1}}p_2(J)^2\Big)^{1/2}
\Big)^2.
\end{align*}
Now let 
\begin{gather*}
a=\sum_{I\in\Subs{A'}{j}}g(I)^2,\quad 
b=\sum_{J\in\Subs{A'}{k-j}}h(J)^2,\\
c=\sum_{I\in\Subs{A'}{j-1}}g(I\cup\{m\})^2,\quad
d=\sum_{J\in\Subs{A'}{k-1-j}}h(J\cup\{m\})^2.
\end{gather*}
In particular, we have 
\begin{align*}
a+c=\sum_{I\in\Subs{A}{j}}g(I)^2,\quad
b+d=\sum_{J\in\Subs{A}{k-j}}h(J)^2.
\end{align*}
The induction hypothesis implies that 
\begingroup
\allowdisplaybreaks
\begin{gather}
\sum_{J\in\Subs{A'}{k}}p(J)^2
\leq\frac{\binomial{n-1}{k}\binomial{k}{j}^2}{
\binomial{n-1}{j}\binomial{n-1}{k-j}}ab
=\frac{(n- k)\binomial{n}{k}\binomial{k}{j}^2}{n
\binomial{n-1}{j}\binomial{n-1}{k-j}}ab,\label{346244}\\
\sum_{J\in\Subs{A'}{k-1}}p_1(J)^2
\leq \frac{\binomial{n-1}{k-1}\binomial{k-1}{j-1}^2}{
\binomial{n-1}{j-1}\binomial{n-1}{k-j}}bc
=\frac{j(n-j)\binomial{n}{k}\binomial{k}{j}^2}{k n
\binomial{n-1}{j}\binomial{n-1}{k-j}}bc,\label{346245}\\
\sum_{J\in\Subs{A'}{k-1}}p_2(J)^2
\leq \frac{\binomial{n-1}{k-1}\binomial{k-1}{j}^2}{
\binomial{n-1}{j}\binomial{n-1}{k-1-j}}ad
=\frac{( k-j)(n- k+j)
\binomial{n}{k}\binomial{k}{j}^2}{k n
\binomial{n-1}{j}\binomial{n-1}{k-j}}ad.\label{346246}
\end{gather}
\endgroup
Hence  
\begin{align*}
\frac{T}{\binomial{n}{k}\binomial{k}{j}^2}
&\leq \frac{1}{\binomial{n}{k}\binomial{k}{j}^2}\Big(
\sum_{J\in\Subs{A'}{k}}p(J)^2
+\Big(\Big(\sum_{J\in\Subs{A'}{k-1}}p_1(J)^2\Big)^{1/2}
+\Big(\sum_{J\in\Subs{A'}{k-1}}p_2(J)^2
\Big)^{1/2}\Big)^2\Big)\\
&\leq \frac{1}{k n\binomial{n-1}{j}\binomial{n-1}{k-j}}
\big( k(n- k)ab+
\big(\sqrt{j(n-j)bc}+\sqrt{( k-j)(n- k+j)ad}\big)^{2}\big)\\
&=\frac{1}{k n\binomial{n-1}{j}\binomial{n-1}{k-j}}
\big( k(n- k)ab+j(n-j)bc+( k-j)(n- k+j)ad+u\big),
\end{align*}
where $u=2\sqrt{j(n-j)( k-j)(n- k+j)abcd}$. 
Now let
\begin{align*}
v=j( k-j)ab +(n-j)(n- k+j)cd,\\
w=( k-j)(n-j)bc +j(n- k+j)ad.
\end{align*}
By using the inequality $2xy\leq x^2+y^2$ for $x,y\in[0,\infty)$,
we derive
\begin{align*}
u&\leq \min\{v,w\}
\leq\frac{k}{n} v+\frac{n- k}{n} w.
\end{align*}
Since all three terms 
\begin{gather*}
k(n-k)+\frac{k}{n}j( k-j), \quad
j(n-j)+\frac{n-k}{n}(k-j)(n-j), \quad\mbox{and}\quad  \\
(k-j)(n-k+j)+\frac{n-k}{n}j(n-k+j)
\end{gather*}
are equal to 
$\frac{k}{n}(n-j)(n- k+j)$, 
we obtain 
\begin{align*}
\frac{T}{\binomial{n}{k}\binomial{k}{j}^2}
&\leq\frac{1}{k n\binomial{n-1}{j}\binomial{n-1}{k-j}}
\big( k(n- k)ab+j(n-j)bc+( k-j)(n- k+j)ad+u\big)\\
&\leq\frac{1}{k n\binomial{n-1}{j}\binomial{n-1}{k-j}}
\frac{k}{n}(n-j)(n- k+j) (ab+bc+ad+cd)\\
&=\frac{1}{\binomial{n}{j}\binomial{n}{k-j}}(a+c)(b+d)
=\Big(\frac{1}{\binomial{n}{j}}
\sum_{I\in\Subs{A}{j}}g(I)^2
\Big)\Big(\frac{1}{\binomial{n}{k-j}}
\sum_{J\in\Subs{A}{k-j}}h(J)^2\Big).\qedhere
\end{align*}
\end{proof}
\begin{proof}[Proof of  Remark \ref{r28759} (Sufficiency).]
It is easily shown that, if one of the conditions 
\ref{r28759.i}--\ref{r28759.v} is valid, 
then equality in \eqref{e717659} holds. It is noteworthy that, if 
$k=n$, then equality in \eqref{e717659} is equivalent to 
the existence of a number $x\in[0,\infty)$  such that 
$g(I)=xh(A\setminus I)$ for all $I\in\Subs{A}{j}$. 
\end{proof}


\begin{proof}[Proof of Remark \ref{r28759} (Necessity).]
We now use induction over $n$ to show that equality in 
\eqref{e717659} implies one of the conditions 
\ref{r28759.i}--\ref{r28759.v}.  For $n=1$, 
the assertion is clear. In the proof for $n\in\NN\setminus\{1\}$, we 
assume the validity of the assertion for $n-1$. 
Let us further assume that, in \eqref{e717659}, equality holds and 
that the conditions \ref{r28759.i}--\ref{r28759.iv} do no hold. 
Then we have $1\leq j<k<n$ and there are $I_0\in\Subs{A}{j}$ and 
$J_0\in\Subs{A}{k-j}$ such that $g(I_0)>0$ and $h(J_0)>0$. The aim is 
to show that \ref{r28759.v} holds. We note  that \eqref{e717659} 
remains the same if we interchange $(g,j)$ and $(h,k-j)$. 
Therefore, if $j=1$ or $j=k-1$, then the 
assertion follows from \citet[Remark~3.1]{MR3916882}. 
So let us additionally assume 
that $2\leq j\leq k-2$. In particular, $n\geq 5$.
Because of the property observed above and since 
$2\leq k-j\leq k-2$, it suffices to show that $g(I)=g(I')$ for all 
$I,I'\in\Subs{A}{j}$. 

In view of the proof of Theorem \ref{t983467}, we see that, for all 
$m\in A$, equality holds in \eqref{346244} and \eqref{346245}, that is
\begin{align}
\sum_{J\in\Subs{A\setminus\{m\}}{k}}p(J)^2
&=\frac{\binomial{n-1}{k}\binomial{k}{j}^2}{
\binomial{n-1}{j}\binomial{n-1}{k-j}}
\Big(\sum_{I\in\Subs{A\setminus\{m\}}{j}}g(I)^2\Big) 
\sum_{J\in\Subs{A\setminus\{m\}}{k-j}}h(J)^2,\label{e1245}\\
\lefteqn{\sum_{J\in\Subs{A\setminus\{m\}}{k-1}}
\Big(\sum_{I\in\Subs{J}{j-1}}g(I\cup\{m\})h(J\setminus I)\Big)^2}
\hspace{3cm}\nonumber\\
&=\frac{\binomial{n-1}{k-1}\binomial{k-1}{j-1}^2}{
\binomial{n-1}{j-1}\binomial{n-1}{k-j}} 
\Big(\sum_{I\in\Subs{A\setminus\{m\}}{j-1}}g(I\cup\{m\})^2\Big)
\sum_{J\in\Subs{A\setminus\{m\}}{k-j}}h(J)^2.\label{e1246}
\end{align}
We note, that \eqref{346246} is not needed  here. In what follows, 
we consider two cases. 
\begin{enumerate}

\item 
Let us first assume that $k<n-1$. 
\begin{enumerate}[(i)]

\item \label{p2356_a.i}
The induction hypothesis and \eqref{e1245} imply 
that, if $m\in A$, $I,I'\in\Subs{A\setminus\{m\}}{j}$, 
$J\in\Subs{A\setminus\{m\}}{k-j}$, $g(I)>0$, and $h(J)>0$, 
then we have $g(I)=g(I')$. 
Here, we had to use condition \ref{r28759.v} of Remark \ref{r28759}, 
since $j\notin\{0,k\}$ and $k<n-1$.

\item Let $I\in\Subs{A}{j}$. We now show that $g(I)=g(I_0)$. 
We note that $\card{A\setminus (I_0\cup J_0)}\geq n-k\geq2$. 
Let $m_1,m_2\in A\setminus (I_0\cup J_0)$ with $m_1\neq m_2$.
 
If there is an $r\in\set{2}$ with $m_r\notin I$, then \ref{p2356_a.i} 
implies that $g(I)=g(I_0)$, since 
$I,I_0\in\Subs{A\setminus\{m_r\}}{j}$, 
$J_0\in\Subs{A\setminus\{m_r\}}{k-j}$, $g(I_0)>0$, and $h(J_0)>0$. 

Let us now assume that $m_1,m_2\in I$. 
Since $\card{A\setminus (J_0\cup I)}\geq 2$, there is 
an $m_3\in A\setminus (J_0\cup I)$. 

If $m_3\notin I_0$, then \ref{p2356_a.i} implies that $g(I)=g(I_0)$
since $I,I_0\in\Subs{A\setminus\{m_3\}}{j}$, 
$J_0\in\Subs{A\setminus\{m_3\}}{k-j}$, 
$g(I_0)>0$, and $h(J_0)>0$.

Let us now assume that $m_3\in I_0$, that is $m_3\neq m_1$. 
Let $I_0'=(I_0\setminus \{m_3\})\cup\{m_1\}$. 
Then \ref{p2356_a.i} implies that $g(I_0')=g(I_0)>0$, since 
$I_0',I_0\in\Subs{A\setminus\{m_2\}}{j}$, 
$J_0\in\Subs{A\setminus\{m_2\}}{k-j}$, $g(I_0)>0$, and $h(J_0)>0$. 
Further, we get $g(I)=g(I_0')=g(I_0)$, since 
$I,I_0'\in\Subs{A\setminus\{m_3\}}{j}$, 
$J_0\in\Subs{A\setminus\{m_3\}}{k-j}$, $g(I_0')>0$, and $h(J_0)>0$. 

Combining the different cases above, we obtain $g(I)=g(I_0)$.

\end{enumerate}

\item
Let us now assume that $k=n-1$. 
\begin{enumerate}[(i)]

\item \label{p2356_b.i}
The induction hypothesis and \eqref{e1245} imply that, 
if $m\in A$, $I\in\Subs{ A\setminus\{m\}}{j}$, $g(I)>0$, and 
$J=(A\setminus\{m\})\setminus I$, then $h(J)>0$. In fact, a number 
$x_m\in[0,\infty)$ exists such that $g(I)=x_mh(J)$. Here condition 
\ref{r28759.iv} of Remark \ref{r28759} has been used.  

\item \label{p2356_b.ii} 
The induction hypothesis and \eqref{e1246}  imply 
that, if $m\in A$, $I,I'\in\Subs{A}{j}$ with 
$m\in I\cap I'\neq\emptyset$, 
$J\in \Subs{A\setminus\{m\}}{k-j}$, $g(I)>0$, and $h(J)>0$, 
then we have $g(I)=g(I')$. 
Here, condition \ref{r28759.v} of Remark \ref{r28759} has bee used, 
since $j-1\notin\{0,k-1\}$ and $k-1<n-1$.

\item \label{p2356_b.iii}
From \ref{p2356_b.i} and \ref{p2356_b.ii}, it follows that, 
if $I,I'\in\Subs{A}{j}$ with $I\cap I'\neq\emptyset$ and
$g(I)>0$, then we have $g(I)=g(I')$. Indeed, since 
$\card{A\setminus I}= n-j\geq n-(k-2)=3$, there is an 
$m_1\in A\setminus I$; from \ref{p2356_b.i} we get $h(J)>0$ for 
$J=(A\setminus\{m_1\})\setminus I$. Let $m_2\in I\cap I'$. 
Then $J\in\Subs{A\setminus \{m_2\}}{k-j}$ and 
\ref{p2356_b.ii} implies $g(I)=g(I')$. 

\item \label{p2356_b.iv}
Since $g(I_0)>0$, \ref{p2356_b.iii} implies 
that $g(I)=g(I_0)$ for all $I\in\Subs{A}{j}$ with 
$I\cap I_0\neq\emptyset$. 

\item It remains to show that 
$g(I)=g(I_0)$ for $I\in\Subs{A}{j}$ with $I\cap I_0=\emptyset$.
Let $m_1\in I_0$, $m_2\in I$, and set 
$I'=(I_0\setminus\{m_1\})\cup\{m_2\}\in\Subs{A}{j}$. 
Since $\card{I_0}=j\geq 2$, we have  
$I'\cap I_0\neq\emptyset$, and \ref{p2356_b.iv} gives $g(I')=g(I_0)>0$.
Because of $I'\cap I\neq\emptyset$ and $g(I')>0$, \ref{p2356_b.iv}
implies that $g(I)=g(I')=g(I_0)$. \qedhere
\end{enumerate}
\end{enumerate}
\end{proof} 

\begin{proof}[Proof of Theorem \ref{c72566}]
We use induction over $\ell$. For $\ell=1$, the assertion follows from 
Theorem \ref{t983467}. In the proof of the assertion for 
$\ell\in\NN\setminus\set{1}$, we assume that the assertion for 
$\ell-1$ is valid. 
Let $T$ denote the left-hand side of the inequality in 
\eqref{e78369888}, that is
\begin{align*}
T
&=\frac{1}{\prod_{s=1}^\ell (\binomial{n_s}{k_s}
\binomial{k_s}{j_s}^2)}
\sum_{J_1\in\Subs{A_1}{k_1}}
\sum_{I_1,I_1'\in\Subs{J_1}{j_1}}
\dots
\sum_{J_{\ell -1}\in\Subs{A_{\ell -1}}{k_{\ell -1}}}
\sum_{I_{\ell -1},I_{\ell -1}'\in\Subs{J_{\ell -1}}{j_{\ell -1}}}\\
&\quad{}\times
\sum_{J_\ell \in\Subs{A_\ell }{k_\ell }}
\sum_{I_\ell \in\Subs{J_\ell }{j_\ell }}
g(I_1,\dots,I_\ell )h(J_1\setminus I_1,\dots,J_\ell\setminus I_\ell)\\
&\quad{}\times
\sum_{I_\ell '\in\Subs{J_\ell }{j_\ell }}g(I_1',\dots,I_\ell ')
h(J_1\setminus I_1',\dots,J_\ell\setminus I_\ell ').
\end{align*}
Using the Cauchy-Schwarz inequality, we obtain  
\begin{align*}
T
&\leq \frac{1}{\prod_{s=1}^{\ell -1}(\binomial{n_s}{k_s}
\binomial{k_s}{j_s}^2)}
\sum_{J_1\in\Subs{A_1}{k_1}}
\sum_{I_1,I_1'\in\Subs{J_1}{j_1}}\dots
\sum_{J_{\ell -1}\in\Subs{A_{\ell -1}}{k_{\ell -1}}}
\sum_{I_{\ell -1},I_{\ell -1}'\in\Subs{J_{\ell -1}}{j_{\ell -1}}}\\
&\quad{}\times\Big(\frac{1}{\binomial{n_\ell }{k_\ell }}
\sum_{J_\ell \in\Subs{A_\ell }{k_\ell }}
\Big(\frac{1}{\binomial{k_\ell }{j_\ell }}
\sum_{I_\ell \in\Subs{J_\ell }{j_\ell }}
g(I_1,\dots,I_\ell )
h(J_1\setminus I_1,\dots,J_\ell \setminus I_\ell )\Big)^2\Big)^{1/2}\\
&\quad{}\times\Big(\frac{1}{\binomial{n_\ell }{k_\ell }}
\sum_{J_\ell \in\Subs{A_\ell }{k_\ell }}
\Big(\frac{1}{\binomial{k_\ell }{j_\ell }}
\sum_{I_\ell '\in\Subs{J_\ell }{j_\ell }}
g(I_1',\dots,I_\ell ')
h(J_1\setminus I_1',\dots,J_\ell \setminus I_\ell ')\Big)^2\Big)^{1/2}.
\end{align*}
Theorem \ref{t983467} implies that 
\begin{align*}
T&\leq \frac{1}{\prod_{s=1}^{\ell -1}(\binomial{n_s}{k_s}
\binomial{k_s}{j_s}^2)}
\sum_{J_1\in\Subs{A_1}{k_1}}
\sum_{I_1,I_1'\in\Subs{J_1}{j_1}}\dots
\sum_{J_{\ell -1}\in\Subs{A_{\ell -1}}{k_{\ell -1}}}
\sum_{I_{\ell -1},I_{\ell -1}'\in\Subs{J_{\ell -1}}{j_{\ell -1}}}\\
&\quad{}\times
\big(\widetilde{g}(I_1,\dots,I_{\ell -1})
\widetilde{h}(J_1\setminus I_1,\dots,J_{\ell -1}\setminus I_{\ell -1})
\big)^{1/2}\\
&\quad{}\times\big(\widetilde{g}(I_1',\dots,I_{\ell -1}')
\widetilde{h}(J_1\setminus I_1',\dots,J_{\ell -1}
\setminus I_{\ell -1}')
\big)^{1/2},
\end{align*}
where, for $I_s\in\Subs{A_s}{j_s}$ and
$J_s\in\Subs{A_s}{k_s-j_s}$, 
$(s\in\set{\ell -1})$,
\begin{align*}
\widetilde{g}(I_1,\dots,I_{\ell -1})
&=\frac{1}{\binomial{n_\ell }{j_\ell }}
\sum_{I_\ell \in\Subs{A_\ell }{j_\ell }}g(I_1,\dots,I_{\ell-1},
I_\ell)^2,\\
\widetilde{h}(J_1,\dots,J_{\ell -1})
&=\frac{1}{\binomial{n_\ell }{k_\ell -j_\ell }}
\sum_{J_\ell \in\Subs{A_\ell }{k_\ell -j_\ell }}
h(J_1,\dots,J_{\ell -1},J_\ell )^2.
\end{align*}
Equivalently, we have 
\begin{align*}
T&\leq \frac{1}{\prod_{s=1}^{\ell -1}\binomial{n_s}{k_s}}
\sum_{J_1\in\Subs{A_1}{k_1}}\dots 
\sum_{J_{\ell -1}\in\Subs{A_{\ell -1}}{k_{\ell -1}}}
\Big(\frac{1}{\prod_{s=1}^{\ell -1}\binomial{k_s}{j_s}}\\
&\qquad\qquad{}\times
\sum_{I_1\in\Subs{J_1}{j_1}}\dots
\sum_{I_{\ell -1}\in\Subs{J_{\ell -1}}{j_{\ell -1}}}
\sqrt{\widetilde{g}(I_1,\dots,I_{\ell -1})
\widetilde{h}(J_1\setminus I_1,\dots,
J_{\ell -1}\setminus I_{\ell -1})}\Big)^2.
\end{align*}
The assertion now follows with the help of the induction hypothesis.
\end{proof}
\apptocmd{\sloppy}{\hbadness 10000\relax}{}{} 

\small 
\let\oldbibliography\thebibliography
\renewcommand{\thebibliography}[1]{\oldbibliography{#1}
\setlength{\itemsep}{0.7ex plus0.5ex minus0.7ex}
} 
\linespread{1.1}
\selectfont
\bibliography{nph_92}
 
\end{document}